\documentclass[11pt]{amsart}
\usepackage{geometry}
\usepackage{amssymb, amsmath, amsthm, color, tikz,enumerate,mathrsfs}
\usepackage{pdfpages}
\usepackage{hyperref}
\usepackage{textcomp}
\usepackage{color}
\usepackage{xcolor}
\usepackage{listings}
\usepackage{todonotes}
\usepackage[style=alphabetic, sorting=nty]{biblatex}
\usepackage{fullpage}

\newtheorem{lemma}{Lemma}[section]
\newtheorem{proposition}[lemma]{Proposition}
\newtheorem{theorem}[lemma]{Theorem}
\newtheorem{corollary}[lemma]{Corollary}

\theoremstyle{definition}
\newtheorem{definition}[lemma]{Definition}
\newtheorem{remark}[lemma]{Remark}

\makeatletter
\newcommand{\setword}[2]{%
  \phantomsection
  #1\def\@currentlabel{\unexpanded{#1}}\label{#2}%
}
\makeatother

\numberwithin{equation}{section}

\newcommand{\Rrr}{\mathbb{R}}

\newcommand{\Zzz}{\mathbb{Z}}
\newcommand{\R}{\Rrr}

\newcommand{\Z}{\Zzz}

\newcommand{\Bf}{\mathcal{B}}
\newcommand{\Cf}{\mathcal{C}}
\newcommand{\Hf}{\mathcal{H}}
\newcommand{\Tc}{\mathscr{T}}
\newcommand{\Tt}{\mathcal{T}}
\newcommand{\Zf}{\mathcal{Z}}
\newcommand{\Pf}{\mathcal{P}}
\newcommand{\Zfe}{\mathcal{Z_{\mathrm{ext}}}}
\newcommand{\Zfep}{\mathcal{Z_{\mathrm{ext}}'}}
\newcommand{\Pfe}{\mathcal{P_{\mathrm{ext}}}}
\newcommand{\fl}{\longrightarrow}

\DeclareMathOperator{\id}{id}
\DeclareMathOperator{\Span}{Span}
\DeclareMathOperator{\Vol}{Vol}
\DeclareMathOperator{\Ker}{Ker}

\newcommand\vanish[1]{}

\makeatletter
\@namedef{subjclassname@2020}{%
  \textup{2020} Mathematics Subject Classification}
\makeatother

\begin{document}

\title{Pizza and $2$-structures}

\author{Richard EHRENBORG, Sophie MOREL and Margaret READDY}

\address{Department of Mathematics, University of Kentucky, Lexington,
  KY 40506-0027, USA.\hfill\break \tt http://www.math.uky.edu/\~{}jrge/,
  richard.ehrenborg@uky.edu.}

\address{ENS de Lyon,
Unit\'e De Math\'ematiques Pures Et Appliqu\'ees,
69342 Lyon Cedex 07,
France.\hfill\break
\tt http://perso.ens-lyon.fr/sophie.morel/,
sophie.morel@ens-lyon.fr.}

\address{Department of Mathematics, University of Kentucky, Lexington,
  KY 40506-0027, USA.\hfill\break
\tt http://www.math.uky.edu/\~{}readdy/,
margaret.readdy@uky.edu.}



\date{\today.}

\begin{abstract}
Let $\mathcal{H}$ be a Coxeter hyperplane arrangement 
in $n$-dimensional Euclidean space.
Assume that the negative of the identity map belongs to the associated Coxeter group~$W$.
Furthermore assume that the arrangement is not of type $A_1^n$.
Let $K$ be a measurable subset of the Euclidean space with finite volume
which is stable by the Coxeter group~$W$ and let $a$ be a point 
such that $K$ contains the convex hull of the orbit of the point~$a$ under the group~$W$.
In a previous article the authors proved the
generalized pizza theorem:
that the alternating sum over the chambers~$T$ of~$\mathcal{H}$ of
the volumes of the intersections $T\cap(K+a)$ is zero.
In this paper we give a dissection proof of this result.
In fact, we lift the identity to an abstract dissection
group to obtain a similar identity that replaces the volume by any
valuation that is invariant under affine isometries.
This includes the cases of all intrinsic volumes.
Apart from basic geometry, the main ingredient is 
a theorem of the authors 
where we relate the alternating sum
of the values of certain valuations over the chambers of a Coxeter arrangement
to similar alternating sums for simpler subarrangements called
$2$-structures
introduced by Herb to study discrete series characters of real
reduced groups.
\end{abstract}

\maketitle

\section{Introduction}

The $2$-dimensional pizza theorem is the following result: Given a disc in the plane,
choose a point on this disc and cut
the disc by $2k$ equally spaced lines passing through the point,
where $k \geq 2$. 
The alternating sum of the areas of the resulting slices is then equal to zero.
This was first proved by Goldberg~\cite{Gol}.  Frederickson
gave a dissection proof~\cite{Fred} based on dissection proofs
of Carter-Wagon in the case $k=2$ (see~\cite{Carter_Wagon}) and
of Allen Schwenk (unpublished) in the cases $k=3,4$. Frederickson deduced
dissection proofs of a similar sharing result for the pizza crust and
of the so-called calzone theorem, which is the analogue of the
pizza theorem for a ball in $\R^3$ that is cut by one horizontal
plane and by $2k$ equally-spaced vertical planes all meeting
at one point in the ball.

To generalize the pizza problem, consider a finite central hyperplane
arrangement~$\Hf$ in~$\R^n$ and fix a base chamber of this arrangement.
Each chamber~$T$ has a sign $(-1)^{T}$ determined by the parity of the number of
hyperplanes separating it from the base chamber. If $K$ is a measurable
subset of~$\R^n$ of finite volume, what can we say about
the pizza quantity $\sum_T (-1)^T\Vol(T\cap K)$, 
where the sum runs over all the chambers $T$ of~$\Hf$?
The original pizza theorem is the case where $n=2$, $\Hf$ has the type of
the dihedral arrangement $I_2(2k)$ 
and $K$ is a disc containing the origin.
The calzone theorem is the case where $n=3$, $\Hf$ has the type of the product
arrangement $I_2(2k)\times A_1$ and $K$ is a ball containing the origin.

The following generalization of the pizza and calzone theorems was
proved in~\cite[Theorem~1.2]{EMR_pizza} by analytic means.
We recently learned that Brailov had proved independently
this result in the case of a ball for the type $B_{n}$ arrangement
using similar methods~\cite{Brailov}.
\begin{theorem}[Ehrenborg--Morel--Readdy]
Let $\Hf$ be a Coxeter arrangement with Coxeter group~$W$
that contains the negative of the identity map, denoted by $-\id$.
Assume that $\Hf$ is not of type $A_{1}^{n}$.
Let $K$ be a set of finite measure that is 
stable by the group~$W$.
Then for every point $a\in\R^n$ such that $K$ contains the convex hull of $\{w(a) : w \in W\}$, 
we have
\[\sum_T (-1)^T\Vol(T\cap(K+a))=0 .\]
\label{theorem_introduction_first}
\end{theorem}
The proof of this result uses an expression for $\sum_T(-1)^T\Vol(T\cap (K+a))$ 
as an alternating sum of pizza quantities
over subarrangements of $\Hf$ of the form
$\{e_1^\perp,\ldots,e_n^\perp\}$ with $(e_1,\ldots,e_n)$ an
orthonormal basis of $\R^n$, in other words, subarrangements
of type $A_1^n$.

In the paper~\cite{EMR}, we study a different sum
$\sum_T (-1)^T\nu(\overline{T})$, where $\nu$ is a valuation
defined on closed convex polyhedral cones of $\R^n$ that takes 
integer values.
Under the same condition that $\Hf$ is a Coxeter arrangement,
we rewrite this quantity as an alternating
sum of similar quantities for certain subarrangements of $\Hf$ that are products
of rank~$1$ and rank~$2$ arrangements~\cite[Theorem~3.2.1]{EMR},
and then deduce an expression for it. These subarrangements,
called {\em $2$-structures}, were introduced
by Herb~\cite{Herb-2S} to study characters of discrete
series of real reductive groups.
In fact, the identity of~\cite[Theorem~3.2.1]{EMR} is valid for
any valuation and its proof uses
only basic properties of Coxeter systems and
closed convex polyhedral cones. 

In this paper we use the setting of $2$-structures 
and~\cite[Corollary~3.2.4]{EMR} (recalled in
Theorem~\ref{thm_pizza_abstract})
to obtain a dissection
proof of the higher-dimensional
pizza theorem of~\cite[Theorem~1.2]{EMR_pizza}
that is independent of the results and methods of~\cite{EMR_pizza}:

\begin{theorem}[Abstract pizza theorem;
see Theorem~\ref{theorem_pizza_dissection}.]
With the notation and hypotheses of Theorem~\ref{theorem_introduction_first},
we have
\[\sum_{T}(-1)^T[\overline{T}\cap(K+a)] = 
\sum_{T}(-1)^T[T\cap(K+a)] =
0,\]
where the brackets denote classes in the abstract dissection
group of Definition~\ref{def_K(V)}.
\label{theorem_introduction_second}
\end{theorem}

As we take into account lower-dimensional sets when defining our abstract
dissection group, this result implies
generalizations of the higher-dimensional pizza theorem to all the
intrinsic volumes when $K$ is convex.

The idea of the proof of Theorem~\ref{theorem_introduction_second}
is the following: by expanding the expression
using $2$-structures, we can
reduce to a sum where each term is a similar expression for
an arrangement that is a product of arrangements
of types~$A_1$ and~$I_2(2^k)$. We then adapt the dissection proof
of Frederickson to an arrangement
of type $I_2(2m)\times\Hf'$. We also explain how to keep track of 
lower-dimensional regions of the dissection.
If our product arrangement contains at least
one dihedral factor, then its contribution is zero, and we immediately
get a dissection proof of the result.
However, if all the product arrangements
that appear are of type $A_1^n$, then their individual contributions are not zero.
We need one extra step in the proof
to show that the contributions cancel.
This uses a slight
refinement of the Bolyai-Gerwien Theorem explained
in Section~\ref{section_Bolyai-Gerwien_plus_epsilon}.

An interesting point to note is that the shape of the pizza plays
absolutely no role in this proof, as long as it has 
the same symmetries as the arrangement and contains
the convex hull of $\{w(a) : w \in W\}$. In particular, we no longer
need to assume that it is measurable and of finite volume.

The plan of the paper is as follows. Section~\ref{section_review}
contains a review of $2$-structures and of the results 
from~\cite{EMR} that we will need. Section~\ref{section_dissection}
contains the statement and proof of the abstract pizza theorem
(Theorem~\ref{theorem_pizza_dissection}),
and Section~\ref{section_Bolyai-Gerwien_plus_epsilon}, as we already
mentioned, contains a Bolyai-Gerwien type result that is needed
in the proof of the abstract pizza theorem.

Let us mention some interesting questions that remain open:
\begin{itemize}
\item[(1)] The paper~\cite{EMR_pizza} proves the pizza
theorem for more general arrangements (the condition is that
the arrangement~$\Hf$ is a Coxeter arrangement
and that the number of hyperplanes is greater than the dimension $n$ and
has the same parity as that dimension), but only in the case of the ball;
see~\cite[Theorem~1.1]{EMR_pizza}. Is it possible to give a dissection
proof of this result?

\item[(2)] Mabry and Deiermann~\cite{MaDe} show that the
two-dimensional pizza theorem does not hold for a dihedral arrangement having
an odd number of lines. More precisely, they determine the sign
of the quantity $\sum_T(-1)^T\Vol(T\cap K)$, where $K$
is a disc containing the origin, and show that it vanishes if and only
if the center of $K$ lies on one of the lines. Their methods are analytic.
As far as we know, there exists no dissection proof of this result
either.
The higher-dimensional case where $\Hf$ is a Coxeter arrangement
and the number of its hyperplanes does not have the same parity as
$n$ also remains wide open.
\end{itemize}

\section{Review of $2$-structures and of the basic identity}
\label{section_review}

Let $V$ be a finite-dimensional real vector space with an inner product
$(\cdot,\cdot)$. 
For every $\alpha\in V$, we denote by $H_\alpha$ the hyperplane
$\alpha^\perp$ and by $s_\alpha$ the orthogonal reflection in 
the hyperplane $H_\alpha$.

We say that a subset $\Phi$ of $V$ is a
\emph{normalized pseudo-root system} if:
\begin{itemize}
\item[(a)] $\Phi$ is a finite set of unit vectors;
\item[(b)] for all $\alpha,\beta\in\Phi$, we have $s_\beta(\alpha)\in\Phi$
(in particular, taking $\alpha=\beta$, we get that $-\alpha\in\Phi$).
\end{itemize}
Elements of $\Phi$ are called \emph{pseudo-roots}.
The \emph{rank} of $\Phi$ is the dimension of its span.

We call such objects {\em pseudo-root systems} to distinguish them from
the (crystallographic) root systems that appear in representation theory.
If $\Phi'$ is a root system then the set 
$\Phi={\{\alpha/ {\|\alpha\|} : \alpha\in\Phi'\}}$ 
is a normalized pseudo-root system. Not
every normalized pseudo-root system arises in this manner;
see for instance the pseudo-root systems of type $H_{3}$ and $H_{4}$.

We say that a normalized pseudo-root system $\Phi$ is \emph{irreducible}
if, whenever $\Phi=\Phi_1\sqcup\Phi_2$ with $\Phi_1$ and 
$\Phi_2$ orthogonal,
we have either $\Phi_1=\varnothing$ or $\Phi_2=\varnothing$. Every normalized
pseudo-root system can be written uniquely as a disjoint union of
pairwise orthogonal
irreducible normalized pseudo-root systems. Irreducible
normalized pseudo-root systems are classified: they are in one
of the infinite families $A_n$ ($n\geq 1$), $B_n/C_n$ ($n\geq 2$),
\footnote{The pseudo-root systems of types $B_n$ and $C_n$ are identical
after normalizing the lengths of the roots.}
$D_n$ ($n\geq 4$), $I_2(m)$ ($m\geq 3$) or one of the exceptional
types $E_6$, $E_7$, $E_8$, $F_4$, $H_3$ or $H_4$, with
types $I_2(3)$ and $A_2$ isomorphic, as well as types
$I_2(4)$ and $B_2$.
(See
\cite[Chapter~5]{GB}
or Table~1 in~\cite[Appendix~A]{BB}.)

We say that a subset $\Phi^+\subset\Phi$ is a
\emph{positive system} if there exists a total ordering
$<$ on the $\R$-vector space $V$ such that $\Phi^+=\{\alpha\in\Phi :
\alpha>0\}$ (see~\cite[Section~1.3]{Hu-Cox}).
The \emph{Coxeter group} of $\Phi$ is the group of isometries $W$ of
$V$ generated by the reflections $s_\alpha$ for $\alpha\in\Phi$.
This group preserves~$\Phi$ by definition of a normalized pseudo-root
system, and it acts simply transitively on the set of positive systems
by~\cite[Section~1.4]{Hu-Cox}. In particular, the Coxeter group $W$ is finite.

Let $E$ be a finite set of unit vectors of $V$ such that $E\cap(-E)=
\varnothing$. The corresponding \emph{hyperplane arrangement} is
the set of hyperplanes $\Hf=\{H_e : e\in E\}$. 
A \emph{chamber} of $\Hf$ is
a connected component
of $V - \bigcup_{e\in E}H_e$;
we denote by $\Tc(\Hf)$ the set of chambers of $\Hf$.
Fix a chamber $T_0$ to be the {\em base chamber}.
For a chamber $T \in \Tc(\Hf)$
we denote by $S(T,T_0)$ the set of
$e \in E$ such that the two chambers $T$ and~$T_0$ are on different
sides of the hyperplane~$H_e$, and define the \emph{sign} of $T$
to be $(-1)^T=(-1)^{|S(T,T_0)|}$.

We say that
$\Hf$ is a \emph{Coxeter arrangement} if it is stable by 
the orthogonal reflections in
each of its hyperplanes. In that case, the set $\Phi=E\cup(-E)$ is a
normalized pseudo-root system. We call its Coxeter group the Coxeter group
of the arrangement. 
The map sending a positive
system $\Phi^+\subset\Phi$ to the set
$\{v\in V : \forall\alpha\in\Phi^+ \: (v,\alpha) > 0\}$
is a bijection from the positive systems in $\Phi$ to the chambers
of $\Hf$. See, for example,
\cite[Chapitre~V \S~4 \textnumero~8 Proposition~9 p.~99]{Bourbaki}
and the discussion following it.
Conversely, if $\Phi\subset V$ is a normalized pseudo-root
system with Coxeter group $W$
and $\Phi^+\subset\Phi$ is a positive system, then
$\Hf=\{H_\alpha : \alpha\in\Phi^+\}$ is a Coxeter hyperplane
arrangement, and in that case we always take the base chamber~$T_0$ to be
the chamber corresponding to $\Phi^+$.

We now define product arrangements. Let $V_1$ and $V_2$ be two
finite-dimensional real vector spaces equipped with inner products, and
suppose that we are given hyperplane arrangements $\Hf_1$ and
$\Hf_2$ on $V_1$ and $V_2$ respectively. We consider the product
space $V_1\times V_2$, where the factors are orthogonal. The {\em product
arrangement} $\Hf_1\times\Hf_2$ is then the arrangement on
$V_1\times V_2$ with hyperplanes $H\times V_2$ for $H\in\Hf_1$ and
$V_1\times H'$ for $H'\in\Hf_2$. 
If $\Hf_1$ is the empty arrangement, then we write
$V_1\times \Hf_2$  instead of the confusing $\varnothing \times\Hf_2$.
Similarly, if $\Hf_2$ is the empty arrangement, we write $\Hf_1\times V_2$.
If the arrangements
$\Hf_1$ and $\Hf_2$ arise from normalized pseudo-root systems
$\Phi_1\subset V_1$ and $\Phi_2\subset V_2$, then 
their product $\Hf_1\times\Hf_2$
arises from the normalized pseudo-root system
$\Phi_1\times\{0\}\cup\{0\}\times\Phi_2\subset V_1\times V_2$.
We also denote this pseudo-root system by $\Phi_1\times\Phi_2$.

The notion of $2$-structures was introduced by Herb
for root systems
to study the characters of discrete series representations;
see, for example, the review article~\cite{Herb-2S}. The definition we
give here is Definition~B.2.1 of~\cite{EMR}. It has been slightly
adapted to work for pseudo-root systems.

\begin{definition}
Let $\Phi$ be a normalized pseudo-root system with Coxeter group $W$.
A \emph{$2$-structure} for $\Phi$ is a subset~$\varphi$ of~$\Phi$
satisfying the following properties:
\begin{itemize}
\item[(a)] 
The subset $\varphi$ is a disjoint union
$\varphi=\varphi_1\sqcup\varphi_2\sqcup\cdots\sqcup\varphi_r$,
where the $\varphi_i$ are pairwise orthogonal subsets
of $\varphi$ and each of them is an irreducible pseudo-root system of type
$A_1$, $B_2$ or $I_2(2^k)$ for $k\geq 3$.
\item[(b)] Let $\varphi^+=\varphi\cap\Phi^+$. If $w\in W$ is such that
$w(\varphi^+)=\varphi^+$ then $\det(w) = 1$.
\end{itemize}
\label{def_2_structure}
We denote by $\Tt(\Phi)$ the set of $2$-structures for $\Phi$.
\end{definition}

\begin{proposition}
Let $\Phi$ be a normalized pseudo-root system with Coxeter group
$W$.
\begin{itemize}
\item[(i)] The group $W$ acts transitively on
the set of $2$-structures $\Tt(\Phi)$.

\item[(ii)] The pseudo-root system $\Phi$ and its $2$-structures
have the same rank if and only if there exists $w\in W$ whose
restriction to $\Span(\Phi)$ is equal to
$-\id_{\Span(\Phi)}$.

\end{itemize}
\label{proposition_2_structures}
\end{proposition}
\begin{proof}
\begin{itemize}
\item[(i)] See the start of Section~4 of~\cite{Herb-2S} and
Proposition~B.2.4 of~\cite{EMR}.

\item[(ii)] For $\Phi$ arising from a
root system $\Phi'$, these two conditions are equivalent
to the fact that $\Phi'$ is spanned by strongly orthogonal roots;
see, for example, the top of page 2559 of~\cite{Herb-DSC}. For
general pseudo-root systems, see the classification of $2$-structures
in Section~B.4 of~\cite{EMR}.
\qedhere
\end{itemize}
\end{proof}

To each $2$-structure $\varphi\subset\Phi$, we can associate a
sign $\epsilon(\varphi)=\epsilon(\varphi,\Phi^+)$
(see the start of Section~5 and Lemma~5.1 of~\cite{Herb-DSC}
and Definition~B.2.8 of~\cite{EMR}).

We next introduce the abstract pizza quantity.
Let $\Hf$ be a central hyperplane arrangement on~$V$.
Let $\Cf_\Hf(V)$ be the set of closed convex polyhedral cones in~$V$
that are intersections of closed half-spaces bounded by hyperplanes~$H$
where $H\in\Hf$,
and let $K_\Hf(V)$ be the quotient of the free abelian group 
$\bigoplus_{K\in \Cf_\Hf(V)}\Z[K]$
on $\Cf_\Hf(V)$ by the relations
$[K]+[K']=[K\cup K']+[K\cap K']$ for all
$K,K'\in \Cf_\Hf(V)$ such that $K\cup K'\in \Cf_\Hf(V)$.
For $K\in \Cf_\Hf(V)$, we still denote
the image of $K$ in $K_\Hf(V)$ by~$[K]$.
\begin{definition}
Suppose that we have fixed a base chamber of $\Hf$.
The \emph{abstract pizza quantity} of~$\Hf$ is
\[P(\Hf)=\sum_{T\in\Tc(\Hf)}(-1)^T[\overline{T}]\in K_\Hf(V).\]
\end{definition}
\begin{remark}
By Lemma~3.2.3 of~\cite{EMR}, we have
\[P(\Hf)=\sum_{T\in\Tc(\Hf)}(-1)^T[T].\]
We use this alternative definition of $P(\Hf)$ in our proofs.
\label{remark_open_chambers}
\end{remark}

The following result is Corollary~3.2.4 of~\cite{EMR}.
It shows how to evaluate the pizza quantity for a Coxeter arrangement
in terms of the associated $2$-structures.
\begin{theorem}
Let $\Phi\subset V$ be a normalized pseudo-root
system. Choose a positive system $\Phi^+\subset\Phi$
and let $\Hf$ be the hyperplane
arrangement $(H_\alpha)_{\alpha\in\Phi^+}$ on $V$ with base chamber
corresponding to $\Phi^+$.
For every $2$-structure $\varphi\in\Tt(\Phi)$, we 
write $\varphi^+=\varphi\cap\Phi^+$ and
we denote
by $\Hf_\varphi$ the hyperplane arrangement $(H_\alpha)_{\alpha\in\varphi^+}$
with base chamber corresponding to $\varphi^+$.
Then we have
\[P(\Hf)=\sum_{\varphi\in\Tt(\Phi)}\epsilon(\varphi)
P(\Hf_\varphi).\]

\label{thm_pizza_abstract}
\end{theorem}

If $\varphi\in\Tt(\Phi)$ then the closures of the chambers of $\Hf_\varphi$ 
are elements of $\Cf_\Hf(V)$, so $P(\Hf_\varphi)$ makes
sense as an element of $K_\Hf(V)$.

\section{A dissection proof of the 
higher-dimensional pizza theorem}
\label{section_dissection}

\begin{definition}
Let $\Cf(V)$ be a nonempty family of
subsets of $V$ that is stable by finite intersections and affine
isometries and such that, if $C\in \Cf(V)$ and
$D$ is a closed affine half-space of $V$, then $C\cap D\in \Cf(V)$.
Furthermore, we assume that $\Cf(V)$ is
closed with respect to Cartesian products, that is,
if $C_{i} \in \Cf(V_{i})$ for $i=0,1$ then 
$C_{0} \times C_{1} \in \Cf(V_{0} \times V_{1})$.
For example, we could take $\Cf(V)$ to be the family of all
convex subsets of $V$, or of all closed (or compact) convex subsets, or of all
convex polyhedra.

We denote by $K(V)$ the quotient of the free abelian group
$\bigoplus_{C\in \Cf(V)}\Z[C]$ on $\Cf(V)$ by the relations:
\begin{itemize}
\item[--] $[\varnothing]=0$;
\item[--]
$[C\cup C']+[C\cap C']=[C]+[C']$ for all $C,C'\in \Cf(V)$ such that
$C\cup C'\in \Cf(V)$;
\item[--]
$[g(C)]=[C]$, for every $C\in \Cf(V)$ and every affine
isometry $g$ of $V$.
\end{itemize}
For $C\in \Cf(V)$, we still denote the image of $C$ in $K(V)$ by
$[C]$. 
\label{def_K(V)}
\end{definition}

\begin{definition}
A \emph{valuation} on $\Cf(V)$ with values in an abelian group $A$ is
a function $\Cf(V)\fl A$ that can be extended to a morphism
of groups $K(V)\fl A$.
\end{definition}

\begin{remark}
Define $\Bf(V)$ to be the relative Boolean algebra
generated by $\Cf(V)$, that is, the smallest collection
of subsets of $V$ that contains $\Cf(V)$ and is closed
under finite unions, finite intersections and set differences.
Groemer's Integral Theorem states that
a valuation on $\Cf(V)$ can be extended
to a valuation on the Boolean algebra $B(V)$; see~\cite{Gr}
and also~\cite[Chapter~2]{Klain_Rota}.
Applying this to the valuation $C \longmapsto [C]$ with values in $K(V)$,
we see that we can make sense of~$[C]$ for any $C\in \Bf(V)$.
For instance, we have
$[C_1 \cup C_2] = [C_1] + [C_2] - [C_1 \cap C_2]$
and
$[C_1 - C_2] = [C_1] - [C_1 \cap C_2]$.
Moreover, if $\Cf(V)$ is the set of all convex polyhedra in
$V$, then $\Bf(V)$ contains all polyhedra (convex or not), and also
half-open polyhedra.
\label{remark_[C]_for_more_general_C}
\end{remark}

Next we have the following straightforward lemma, whose proof we omit, 
which states that the class symbol is well-behaved with respect
to Cartesian products.
\begin{lemma}
The two class identities 
$[C_{0}] = [D_{0}]$ and $[C_{1}] = [D_{1}]$
in 
$K(V_{0})$ and~$K(V_{1})$, respectively,
imply that
$[C_{0} \times C_{1}] = [D_{0} \times D_{1}]$
in
$K(V_{0}) \times K(V_{1})$.
\label{lemma_product}
\end{lemma}

Let $\Hf$ be a central hyperplane arrangement on $V$ with fixed
base chamber.
If $K\in \Cf(V)$,
we have a morphism of groups $e_K:K_\Hf(V)\longrightarrow K(V)$ 
induced by the map $\Cf_\Hf(V) \longrightarrow \Cf(V)$, $C\longmapsto C\cap K$.

We denote by $P(\Hf,K)$ 
the image of $P(\Hf)$ by this morphism $e_K$;
in other words, we have
\begin{align*}
P(\Hf,K)
& =
\sum_{T\in\Tc(\Hf)}(-1)^T[\overline{T}\cap K].
\end{align*}
By Remark~\ref{remark_open_chambers}, we also have
\begin{align*}
P(\Hf,K)
& =
\sum_{T\in\Tc(\Hf)}(-1)^T[T\cap K].
\end{align*}

We state the main theorem of this paper.
First for $u,v \in V$ define the half-open line segment
$(u,v]$ by $\{(1-\lambda) u + \lambda v : 0 < \lambda \leq 1\}$.
Our main result is the following:
\begin{theorem}
[The Abstract Pizza Theorem]
Let $\Hf$ be a Coxeter hyperplane arrangement with Coxeter group~$W$
in an $n$-dimensional space $V$ such that $-\id_V\in W$.
Let $K\in \Cf(V)$ and $a\in V$. Suppose that $K$ is stable by the group~$W$
and contains the convex hull of the set $\{w(a) : w\in W\}$.
\begin{enumerate}
\item If $\Hf$ is not of type $A_1^{n}$, 
we have $P(\Hf,K+a)=0$ in $K(V)$.

\item If $\Hf$ has type $A_1^{n}$, $\Phi$ is the normalized
pseudo-root system corresponding to~$\Hf$ and 
$\Phi^+=\{e_1,\ldots,e_n\}$ where $\Phi^+\subset\Phi$ is the positive
system corresponding to the base chamber of~$\Hf$,
then the following identity holds:
\begin{align}
\label{equation_half_open}
P(\Hf,K+a)
& =
\left[\prod_{i=1}^n(0,2(a,e_i)e_i]\right] .
\end{align}
\end{enumerate}
\label{theorem_pizza_dissection}
\end{theorem}
Here we are using Remark~\ref{remark_[C]_for_more_general_C} to make
sense of the right-hand side of equation~\eqref{equation_half_open}.

The conditions on $K$ are satisfied if for example $K$ is convex,
contained in $\Cf(V)$, stable by $W$ and $0\in K+a$. Indeed, the last condition
implies that $-a\in K$; as
$-\id_V\in W$ by assumption, this in turns implies that $a\in K$,
hence that $K$ contains the convex hull of the set $\{w(a) : w \in W\}$.

We will give the proof of Theorem~\ref{theorem_pizza_dissection} at the end of the section. This proof does not use Theorem~1.2 of~\cite{EMR_pizza}, so we obtain
a new proof of that result.

Let $V_0,\ldots,V_n$ denote the intrinsic volumes on $V$
(see \cite[Section~4.2]{Schneider}).

\begin{lemma}
Let $(v_1,\ldots,v_k)$ be an orthogonal family of vectors in $V$. Then
\[V_i((0,v_1]\times\ldots\times(0,v_k])=0\]
for $0\leq i\leq k-1$.
\label{lemma_intrinsic_volume_half_open_parallelotope}
\end{lemma}

\begin{proof}
By Lemma~14.2.1 of~\cite{Schneider-Weil}
or
Proposition~4.2.3 of~\cite{Klain_Rota}, it suffices to prove that,
if $a<b$ are real numbers, the $0$th intrinsic volume of the
half-open segment $(a,b]\subset\Rrr$ is $0$. As the $0$th intrinsic
volume is the Euler-Poincar\'e characteristic with compact support, this is clear.
\end{proof}

\begin{corollary}
We keep the notation and hypotheses of Theorem~\ref{theorem_pizza_dissection}.
If $\Hf$ is not of type $A_1^n$, we have
\begin{equation}
\sum_{T\in\Tc(\Hf)}(-1)^T V_i(T\cap(K+a))=0
\label{equation_intrinsic_volumes}
\end{equation}
for every $0 \leq i \leq n$, where $K$ is assumed to be convex if
$i\not=n$.
If $\Hf$ has type $A_1^n$ and $K$ is convex
then equation~\eqref{equation_intrinsic_volumes}
holds for $0 \leq i \leq n-1$.
\end{corollary}
\begin{proof}
If $\Hf$ is not of type $A_1^n$, then equation~\eqref{equation_intrinsic_volumes}
actually holds for any valuation on $\Cf(V)$ that is invariant under
the group of affine isometries; this includes the intrinsic volumes.

Suppose that $\Hf$ is of type $A_1^n$. Then we know that
equation~\eqref{equation_half_open} holds. The result then follows
from Lemma~\ref{lemma_intrinsic_volume_half_open_parallelotope}.
\end{proof}

\begin{remark}
Theorem~\ref{theorem_pizza_dissection} immediately implies
generalizations to our higher-dimensional case
of the ``thin crust'' and ``thick crust'' results
of Confection~3 and Leftovers~1 of~\cite{MaDe}
for an even number of cuts.

We obtain the ``thin crust'' result by evaluating the
$(n-1)$st intrinsic volume on $P(\Hf,K+a)$. Note that
this result holds for a pizza of any (convex) shape and even
in the case where we only make $n$ cuts, where $n$ is the
dimension.

To generalize the ``thick crust'' result, consider two sets
$K\subset L$ stable by $W$ and
in $\Cf(V)$. If $a\in V$
is such that $K$ contains the convex hull of the set $\{w(a) : w\in W\}$,
then
\[P(\Hf,(L - K)+a)=P(\Hf,L+a)-P(\Hf,K+a)=0,\]
so in particular
\[\sum_{T\in\Tc(\Hf)}(-1)^T\Vol(T\cap((L - K)+a))=0.\]
The case where $K$ and $L$ are balls with the same center
is the ``thick crust'' result.
\label{remark_crust}
\end{remark}

We now state and prove some lemmas that will be used in 
the proof of Theorem~\ref{theorem_pizza_dissection}.

\begin{lemma}
Let $\Hf_i$ be a hyperplane arrangement on $V_i$ for $i=0,1$.
Assume furthermore that $\Hf_1 = \{H_e\}_{e \in E_1}$ has type $A_1^r$
and $\dim(V_1) = r$.
Let $E_{1} = \{e_1, \ldots, e_r\}$ be the index set of $\Hf_{1}$.
Let $\Hf$ and $V$ be the Cartesian products $\Hf_0 \times \Hf_{1}$ and
$V_0 \times V_1$, respectively.
Then for every $K\in \Cf(V)$
that is stable under the orthogonal reflections in the hyperplanes
$V_0 \times H_{e_1}, \ldots, V_0 \times H_{e_r}$ and for every $a\in V_1$, if
$L=K+a$, we have the identity
\begin{align*}
P(\Hf,L)
& =
P\left(\Hf_0\times V_1,L\cap (V_0\times(0,2(a,e_1)e_1]\times\cdots\times(0,2(a,e_r)e_r])\right),
\end{align*}
where $\Hf_0\times V_1$ is the product of
$\Hf_0$ and the empty hyperplane arrangement on $V_1$.
\label{lemma_Cartesian_product_with_A_1}
\end{lemma}

\begin{figure}
\begin{center}
\begin{tikzpicture}[scale = 0.7,domain=0:90]
\draw[->,thick] (-2.5,0) -- (11,0) node[above] {\small $V_{0} = H_{e_1}$};
\draw[->,thick] (0,-3.5) -- (0,6) node[right] {\small $V_{1}$};
\filldraw (3,1) circle (2pt);
\draw[-] (-2.5,2) -- (11,2) node[above] {\small $H_{e_1}+2(a,e_1)e_1$};
\draw[-,dotted] (-2.5,1) -- (11,1)  node[above] {\small $H_{e_1}+(a,e_1)e_1$};
\draw[-] (4.5,{-(4^3-(1.5)^3)^(1/3)+1}) -- (4.5,{(4^3-(1.5)^3)^(1/3)+1});
\draw[-] (5.5,{-(4^3-(2.5)^3)^(1/3)+1}) -- (5.5,{(4^3-(2.5)^3)^(1/3)+1});
\node at (5,0.3) {\small $T$};
\draw[] plot ({3+4*cos(\x)^(2/3)},{1+4*sin(\x)^(2/3)});
\draw[] plot ({3+4*cos(\x)^(2/3)},{1-4*sin(\x)^(2/3)});
\draw[] plot ({3-4*cos(\x)^(0.5)},{1+4*sin(\x)^(0.5)});
\draw[] plot ({3-4*cos(\x)^(0.5)},{1-4*sin(\x)^(0.5)});
\end{tikzpicture}
\end{center}
\caption{A schematic sketch of $V_0 \times V_1$
for the proof of Lemma~\ref{lemma_Cartesian_product_with_A_1}.}
\label{figure_H_cross_A_1}
\end{figure}
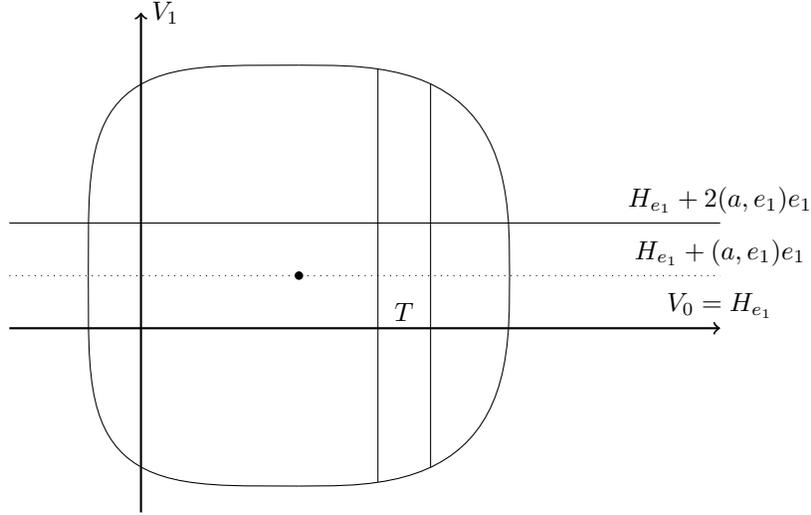

\begin{proof}
By a straightforward induction, we may assume that $r=\dim V_1=1$. 
Also, after changing the sign of~$e_1$, we may assume that
$(a,e_1)\geq 0$. See Figure~\ref{figure_H_cross_A_1} for a sketch
of the situation.
Let~$T$ be a chamber of the arrangement $\Hf_0$.
The classes of the two regions
$(T\times \Rrr_{> 0}e_1) \cap L$
and
$(T\times \Rrr_{< 0}e_1) \cap L$ of~$\Hf$
occur with opposite signs in the pizza quantity
$P(\Hf, L)$.
Note that the region
$(T\times \Rrr_{> 2(a,e_1)}e_1) \cap L$
is the orthogonal reflection of the region
$(T \times \Rrr_{< 0}e_1) \cap L$
in the affine hyperplane~$H_{e_1} + (a,e_1)e_1=H_{e_1}+a$.
Hence these regions have the same class in $K(V)$ which cancels
in the pizza quantity $P(\Hf,L)$, and the class
of the region $(T \times (0,2(a,e_1)e_1]) \cap L=
(T\times V_1)\cap L\cap(V_0\times(0,2(a,e_1)e_1])$
remains. As the map $T\longmapsto T\times V_1$ is a sign-preserving
bijection from $\Tc(\Hf_0)$ to $\Tc(\Hf_0\times V_1)$, this completes
the proof.
\vanish{
The proof for $P(\Hf,L)$ is similar, except that we
need to consider all the faces $F$ of $\Hf_0$, and that
$F\times\{0\}$ will also be a face of $\Hf_0\times\Hf_1$
(which explains why we obtain $[0,2(a,e_i)e_i]$ and not
$(0,2(a,e_i)e_i]$).
}
\end{proof}

We now consider the case of a hyperplane arrangement that is the
product of a $2$-dimensional dihedral arrangement and another arrangement.
Suppose that $V=V_0\times V_1$, where the factors are orthogonal, and
that $\Hf=\Hf_0\times\Hf_1$, where $\Hf_i$ is a hyperplane arrangement in
$V_i$. Suppose also that $\dim V_0 = 2$
and that $\Hf_0$ is an arrangement of type $I_2(2m)$
with Coxeter group $W_0$ where $m\geq 2$.
We view $W_0$ as a group of isometries of $V$ by making $w\in W_0$ act on
$V=V_0\times V_1$ by $w\times\id_{V_1}$. We also choose a family~$\Cf(V)$
as in Definition~\ref{def_K(V)}.

Let $a\in V_0$. We will describe a dissection of $V_0$.
The case where $m=4$ is shown in Figure~\ref{figure_dissection_dihedral}.
We call $L_0,\ldots,L_{2m-1}$ the lines of $\Hf_0$ (numbered so that
the angle between $L_0$ and $L_i$ is an increasing function of $i$)
and we assume that the point $a$ is in a chamber between
$L_{m-1}$ and $L_{m}$. Choose a closed half-space $D$ bounded by
$L_0$ and containing $a$ (this choice is unique if $a\not\in L_0$).
Then, for $0\leq i\leq 2m-1$, we denote by $T_i$ the unique chamber
of 
$\Hf$ contained in $D$ and
with boundary contained in $L_i\cup L_{i+1}$. We assume that
$(-1)^{T_0}=1$ for concreteness. The point $a$ is in the closure
of the chamber~$T_{m-1}$.

We write $\Tc_+=\{T\in\Tc(\Hf_0) : (-1)^T=1\}$ and
$\Tc_-=\{T\in\Tc(\Hf_0) : (-1)^T=-1\}$.
Let $W_a$ be the group of affine isometries generated by the
orthogonal reflections in the lines $L+a$, for $L\in\Hf_0$. We take
$R_0(a)$ to be the convex hull of the points $w(0)$ for $w\in W_a$.
This is the shaded polygon on Figure~\ref{figure_dissection_dihedral},
where the darker slices are the intersections with the closures of
chambers in~$\Tc_+$. 
We have the inclusion
$R_0(a)\subset\bigcup_{i=0}^{2m-2}\overline{T}_i$.
Finally we set
\[R_{0,\pm}(a)=R_{0}(a) \cap \bigcup_{T\in\Tc_{\pm}}T.\]

\begin{lemma}
The following three identities hold in $K(V_{0} \times V_{1})$:
\begin{enumerate}
\item
Let $K \in \Cf(V)$ such that $K$ is stable by $W_0$ and let $L=K+a$. Then
\[
\sum_{T\in\Tc} (-1)^{T} [L \cap ((T - R_0(a))\times V_1)]
=
0 . \]

\item
Let $K \in \Cf(V)$ such that $K$ is stable by $W_0$ and let $L=K+a$ then
\[P(\Hf,L)=
P(V_0\times\Hf_1,L\cap(R_{0,+}(a)\times V_1))-
P(V_0\times\Hf_1,L\cap(R_{0,-}(a)\times V_1)).\]

\item If $K_1\subset V_1$ is such that $R_0(a)\times K_1\in \Cf(V)$,
then
\[[R_{0,+}(a)\times K_1]=[R_{0,-}(a)\times K_1].\]
\end{enumerate}
\label{lemma_dissection_proof_dihedral}
\end{lemma}

\begin{figure}
\centering
\includegraphics[width=\textwidth]{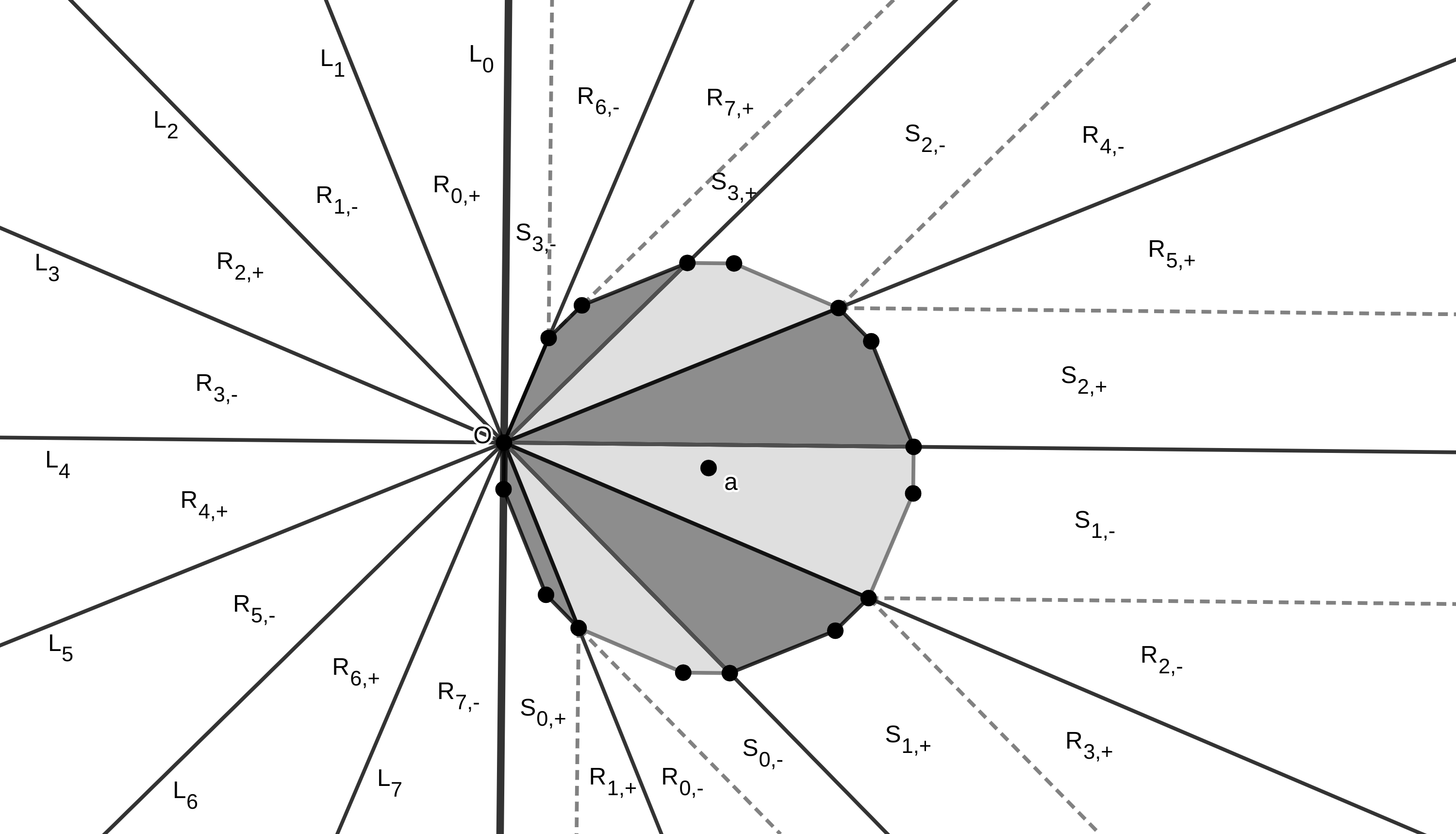}
\caption{A picture of the regions
in the proof of Lemma~\ref{lemma_dissection_proof_dihedral}(i) for
the case of $I_2(8)$.}
\label{figure_dissection_dihedral}
\end{figure}
\begin{proof}
We begin by proving~(i).
For $0\leq i\leq 2m-1$, we denote by $R_{i,\pm}$ the unique
chamber of~$\Hf$ not contained in $D$ and
with boundary contained in $L_i\cup L_{i+1}$, that is, the image
of~$T_i$ by the symmetry with center $0$; 
we write $R_{i,+}$ if
this chamber has sign~$+1$, or equivalently if $i$ is even, and
$R_{i,-}$ if this chamber has sign $-1$, or equivalently if
$i$ is odd. For every $0 \leq j \leq m-1$, we denote
by $R_{2j,-}$ and $R_{2j+1,+}$ the orthogonal reflection 
in the line $L_{2j+1}^\perp+a$ of 
$R_{2j,+}$ and $R_{2j+1,-}$, respectively.
Note that $R_{2j,-}\subset T_{2j+1}$ and $R_{2j+1,+}\subset T_{2j}$.
For $0 \leq j \leq m-1$ again, we denote by~$S_{j,+}$
the interior of $T_{2j} -(R_0(a)\cup R_{2j+1,+})$
and by $S_{j,-}$ the interior of
$T_{2j+1} - (R_0(a)\cup R_{2j,-})$.
Then $T_{2j} - R_0(a)$ is the disjoint union of
$R_{2j+1,+}$, $S_{j,+}$ and an open ray $D_{2j}$ starting at
an extremal point of $R_0(a)$ (the image of $0$ by the orthogonal
reflection in the line $L_{2j+1}^\perp+a$)
and parallel to~$L_{2j}$.
Similarly $T_{2j+1} - R_0(a)$ is the disjoint union
of $R_{2j,-}$, $S_{j,-}$ and an open ray $D_{2j+1}$ starting
at the same extremal point of $R_0(a)$ and parallel to $L_{2j+2}$. 
See Figure~\ref{figure_dissection_dihedral} for the case $m=4$,
where the rays $D_{2j}$ and $D_{2j+1}$ are dashed.

The union $\bigcup_{T\in\Tc_+}T$ is equal to the disjoint union
of the set~$R_{0,+}(a)$,
the regions~$R_{i,+}$ for $0\leq i\leq 2m-1$,
the regions~$S_{j,+}$ for $0\leq j\leq m-1$ and
the rays $D_{2j}$ for
$0\leq j\leq m-1$.
On the other hand, the union
$\bigcup_{T\in\Tc_-}T$ is equal to the disjoint union
of the set~$R_{0,-}(a)$,
the regions~$R_{i,-}$ for $0\leq i\leq 2m-1$,
the regions~$S_{j,-}$ for $0\leq j\leq m-1$ and
the rays~$D_{2j+1}$ for $0\leq j\leq m-1$.
Consider the following four observations:
\begin{itemize}
\item[--]
For $0\leq i\leq 2m-1$ the region~$R_{i,-}$ is the image
of $R_{i,+}$ by the orthogonal reflection in the affine line
$L_{2\lfloor i/2 \rfloor+1}^\perp+a$.
\item[--]
For $0\leq j\leq m-1$ 
the region $S_{j,-}$ is the image
of $S_{j,+}$ by the rotation with center $a$ and angle~${\pi}/{m}$.
\item[--]
For $0\leq j\leq m-2$ the ray $D_{2j+3}$ is the image
of the ray $D_{2j}$ by the rotation with center~$a$ and
angle ${2\pi}/{m}$.
\item[--]
The ray $D_{2m-2}$ is the image of $D_1$ by the orthogonal
reflection in the affine line $L_m+a$.
\end{itemize}
Each of them is of the form: the set~$X$ is the image of the set~$Y$ under
an affine isometry~$g$ belonging to the group~$W_a$.
Since the set $L = K+a$ is invariant under $g$, we obtain that
the set $L \cap (X \times V_1)$ is the image of $L \cap (Y \times V_1)$, and
hence that $[L \cap (X \times V_1)] = [L \cap (Y \times V_1)]$.
Statement~(i) follows by summing over all pairs of sets $X$ and $Y$.

Next we prove~(ii).
There is a bijection
$\Tc(\Hf_0)\times\Tc(\Hf_1)\stackrel{\sim}{\longrightarrow}\Tc(\Hf)$
where
$(T,T') \longmapsto T\times T'$ and $(-1)^{T\times T'}=(-1)^{T}(-1)^{T'}$
for all $T\in\Tc(\Hf_0)$ and $T'\in\Tc(\Hf_1)$. Hence
\begin{align}
P(\Hf,L) & = 
\sum_{T'\in\Tc(\Hf_1)}
\sum_{T\in\Tc(\Hf_0)}
(-1)^{T} (-1)^{T'} [L \cap (T\times{T'})] .
\label{equation_Wednesday}
\end{align}
Fix $T'\in\Tc(\Hf_1)$ for a moment.
The fact that $K$ is stable by~$W_{0}$ implies that
$K \cap (V_0 \times {T'})$
is also stable by~$W_{0}$.
Hence applying statement~(i) to the set
$(K  \cap (V_0 \times {T'})) + a 
= L  \cap (V_0 \times {T'})$ yields
\begin{align}
\sum_{T\in\Tc(\Hf_0)} (-1)^{T} [L \cap ((T - R_0(a))\times {T'})]
& =
0 .
\label{equation_fixed_T'}
\end{align}
Multiplying 
equation~\eqref{equation_fixed_T'} 
with the sign $(-1)^{T'}$,
summing over all $T'\in\Tc(\Hf_1)$,
and subtracting the result from
equation~\eqref{equation_Wednesday}
yields
\begin{align*}
P(\Hf,L)
& = 
\sum_{T'\in\Tc(\Hf_1)}
\sum_{T\in\Tc(\Hf_0)}
(-1)^{T} (-1)^{T'} [L \cap ((T \cap R_0(a)) \times {T'})] \\
& =
\sum_{T'\in\Tc(\Hf_1)}(-1)^{T'}
[L\cap(R_{0,+}(a)\times T')]-
\sum_{T'\in\Tc(\Hf_1)}(-1)^{T'}[L\cap(R_{0,-}(a)\times{T'})]\\
& =
P(V_0\times\Hf_1,L\cap(R_{0,+}(a)\times V_1))
-
P(V_0\times\Hf_1,L\cap(R_{0,-}(a)\times V_1)).
\end{align*}

Finally we consider~(iii).
By Lemma~\ref{lemma_product},
it suffices to show that, if we take $\Cf(V_0)$ to be the set
of convex polygons in~$V_0$, then $[R_{0,+}(a)]=[R_{0,-}(a)]$ in
$K(V_0)$. This follows from Corollary~\ref{cor_Bolyai-Gerwien_plus_epsilon}
and from the fact that the intrinsic volumes of $R_{0,+}(a)$ and $R_{0,-}(a)$
are equal (which is an easy calculation), but we also give
a direct proof.
We consider the following dissection of the polygon~$R_0(a)$;
see Figure~\ref{figure_dissection_dihedral_2} for the case $m=4$.
For $0\leq i\leq 2m-1$ let $P_i$ be the image
of $0$ by the orthogonal reflection in the line $L_i^\perp+a$; note that
$P_i$ is a boundary point of $R_0(a)$, and that it is on $L_i$.
We describe the pieces of the dissection of $R_0(a)$:
\begin{itemize}
\item[--]
For $1 \leq i \leq m-2$ consider the pair of isosceles
triangles $B_{i,+}$ and $B_{i,-}$ that have 
one side equal to the segment
$[0,P_{2i}]$, angles equal to ${\pi}/{2m}$ at the vertices $0$ and $P_i$, and
such that $B_{i,\pm}$ is in a chamber with sign $\pm 1$; in
other words, the triangle $B_{i,-}$ is in the chamber~$T_{2i-1}$, and
$B_{i,+}$ is in the chamber $T_{2i}$.

\item[--]
Let $B_{0,+}$ be the isosceles triangle contained in $T_0$ with
one side equal to the segment
$[0,P_0]$ and angles equal to ${\pi}/{2m}$ at the vertices $0$ and $P_0$.

\item[--]
Consider the isosceles triangle contained in $T_{2m-3}$ with
one side equal to the segment
$[0,P_{2m-2}]$ and angles equal to ${\pi}/{2m}$ 
at the vertices $0$ and $P_{2m-2}$;
this splits into an isosceles triangle $B_{0,-}$ congruent
to $B_{0,+}$ and an isosceles trapezoid $B_{m-1,-}$ having one edge equal
to $[0,P_{2m-2}]$.

\item[--]
Let $B_{m-1,+}$ be the image of $B_{m-1,-}$ by the orthogonal reflection in
the line $L_{2m-2}$; then $B_{m-1,+}$ is contained in the chamber
$T_{2m-2}$.
\end{itemize}

\begin{figure}
\centering
\includegraphics[width=\textwidth]{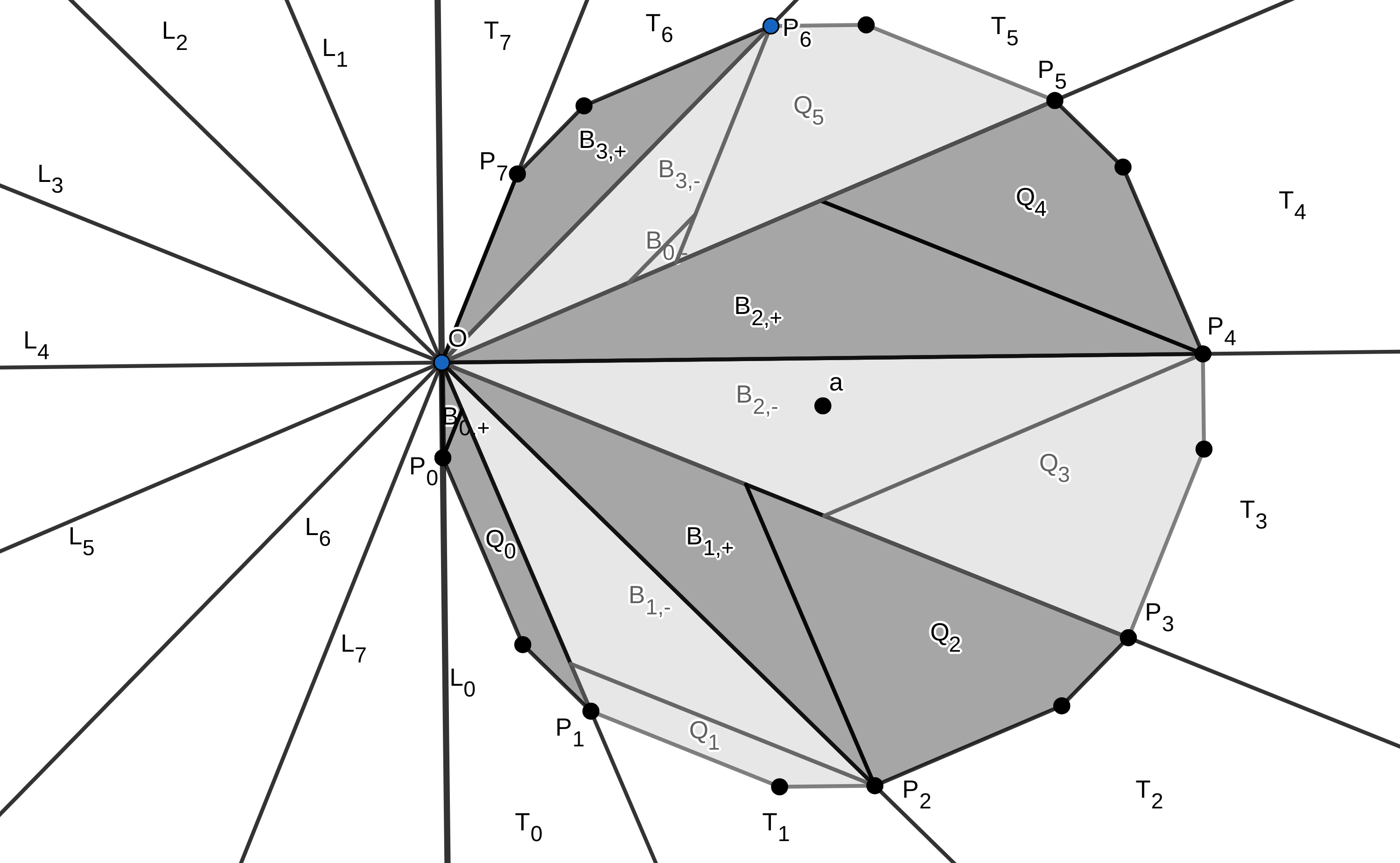}
\caption{
The dissection of the polygon $R_0(a)$ in
the case of $I_2(8)$. 
The darker shaded regions are the intersections
with the chambers which have a positive sign.}
\label{figure_dissection_dihedral_2}
\end{figure}

To finish the dissection of
$R_0(a)$, we note that, for $0\leq i\leq 2m-3$, we still have
a quadrilateral piece $Q_i$ left over in $T_i\cap R_0(a)$.
Then for $1 \leq i \leq 2m-3$ the quadrilateral
$Q_{i}$ is the image of~$Q_{i-1}$ by the rotation with center $a$ and
angle~${\pi}/{m}$. 
Indeed, this is true for the intersections of these
quadrilaterals with the boundary of $R_0(a)$ (which consist of two
edges with endpoints $P_{i-1}$ and~$P_{i+1}$), 
and it is easy to calculate the angles at the vertices of
this intersection and to see that they correspond:
If $1\leq i\leq 2m-2$ and $i$ is even, respectively odd, the angle of $Q_{i-1}$ at
$P_i$ is ${(i-1)\pi}/{2m}$, respectively ${i\pi}/{2m}$. 
If $0 \leq i \leq 2m-3$ and $i$ is even, respectively odd, the angle
of $Q_i$ at $P_i$ is ${(2m-2-i)\pi}/{2m}$,
respectively ${(2m-1-i)\pi}/{2m}$.

This is Frederickson's dissection; see pages 28--31 of~\cite{Fred}.
That paper is only interested in giving a dissection proof that
the alternating sum of the areas is equal to zero. Hence it can safely
ignore line segments of area zero, whereas we are proving
an identity in $K(V_{0})$ and have to be careful with all regions,
including lower dimensional ones.

The set $R_{0,+}(a)$ is the disjoint union of the following subsets:
\begin{itemize}
\item[--]
For $0\leq i\leq m-2$ let $B'_{i,+}$ be the intersection of the triangle
$B_{i,+}$ with the interior of the chamber containing $B_{i,+}$.
In other words,
$B'_{i,+}$ is the union of the interior of $B_{i,+}$ and the relative interior of
one of its two equal sides.

\item[--]
Let $B'_{m-1,+}$ be the intersection of the trapezoid $B_{m-1,+}$ and
$T_{2m-2}$.
That is, $B'_{m-1,+}$ is the union of the interior of
$B_{m-1,+}$ and the intersection of its boundary with the boundary
of $R_0(a)$, minus the two extremal points of this intersection.

\item[--]
For $0\leq j\leq m-2$ let $Q'_{2j}$ be the intersection of the quadrilateral
$Q_{2j}$ and $T_{2j}$. 
That is, $Q'_{2j}$ is the union of the interior of
$Q_{2j+1}$ and the intersection of its boundary with the boundary
of $R_0(a)$, minus the two extremal points of this intersection.
\end{itemize}
As for the set $R_{0,-}(a)$, it is the disjoint union of the
following subsets:
\begin{itemize}
\item[--]
Let $B'_{0,-}$ be the union of the interior of the triangle $B_{0,-}$ and
the relative interior of the side that it shares with~$Q_{2m-3}$.

\item[--]
For $1\leq i\leq m-2$ let $B'_{i,-}$ be the intersection of the triangle
$B_{i,-}$ with the interior of the chamber containing $B'_{i,-}$.
That is, $B'_{i,-}$ is the union of the interior of $B_{i,-}$ and the relative interior of
one of its two equal sides.

\item[--]
Let $B'_{m-1,-}$ be the intersection of the trapezoid $B_{m-1,-}$ and $T_{2m-3}$.

\item[--]
For $0\leq j\leq m-2$ let $Q'_{2j+1}$ be the intersection of the quadrilateral
$Q_{2j+1}$ and $T_{2j+1}$.
That is, $Q'_{2j+1}$ is the union of the interior of
$Q_{2j+1}$ and the intersection of its boundary with the boundary
of $R_0(a)$, minus the two extremal points of this intersection.
\end{itemize}
We obtain that $[R_{0,+}(a)]=[R_{0,-}(a)]$ since the regions
$B'_{i,+}$ and $B'_{i,-}$ are isometric for every
$0 \leq i \leq m-1$, as are the regions
$Q_{2j}$ and $Q_{2j+1}$ for $0 \leq j \leq m-2$.
\end{proof}

\begin{lemma}
Suppose that we have $V=V_1^{(1)}\times\cdots\times V_1^{(r)}\times V_2^{(1)}\times
\cdots\times V_2^{(s)}$, where the factors of the product are pairwise
orthogonal, and that $\Hf$ is a product
$\Hf_1^{(1)}\times\cdots\times\Hf_1^{(r)}\times\Hf_2^{(1)}\times\cdots
\times\Hf_2^{(s)}$, where each $\Hf_i^{(j)}$ is a hyperplane arrangement
on $V_i^{(j)}$. Suppose further that:
\begin{itemize}
\item[(a)] If $1 \leq j \leq r$ then $V_1^{(j)}$ is $1$-dimensional, and
we have a unit vector $e^{(j)}$ in $V_1^{(j)}$
yielding the hyperplane arrangement $\Hf_1^{(j)} = \{0\}$.

\item[(b)] If $1 \leq j \leq s$ then $V_2^{(j)}$ is $2$-dimensional, and
the arrangement $\Hf_2^{(j)}$ is of type $I_2(2m^{(j)})$ for some
$m^{(j)} \geq 2$.
\end{itemize}
Let $a\in V$ and $K\in \Cf(V)$.
Suppose that $K$ is stable by the Coxeter group $W$ and contains the convex
hull of the set $\{w(a) : w\in W\}$.
Then the following two statements hold:
\begin{itemize}
\item[(i)] If $s\geq 1$ we have $P(\Hf,K+a)=0$ in $K(V)$.

\item[(ii)] If $s=0$ we have in $K(V)$ the identity
\begin{align*}
P(\Hf,K+a)
& =
\left[(0,2(a,e^{(1)})e^{(1)}]\times\cdots\times(0,2(a,e^{(r)})e^{(r)}]\right] .
\end{align*}
\end{itemize}
\label{lemma_pizza_for_A_1^rxI_2^s}
\end{lemma}

\begin{proof}
Let $L=K+a$. Since $-\id_V\in W$ we have $-a\in K$ and so $0\in L+a$;
also, the sets $\{w(a) : w\in W\}$ and $\{w(-a) : w\in W\}$ are
equal.
Let $W_a$ be the group of affine isometries of
$V$ generated by the orthogonal reflections in the hyperplanes
$a+H$, for $H\in\Hf$. The conditions on $K$ imply that $L$ is
stable by $W_a$ and contains the convex hull of the set
$\{u(0) : u\in W_a\}=\{w(-a)+a : w\in W\}$.

Write $a=(a_1^{(1)},\ldots,a_1^{(r)},a_2^{(1)},\ldots,a_2^{(s)})$, with
$a_i^{(j)}\in V_i^{(j)}$. 
For $1 \leq i \leq r$ 
let $S^{(i)}$ denote the half-open line-segment $(0, 2(a,e^{(i)}) e^{(i)}]$.
For $1 \leq j \leq s$
we consider subsets 
$R_{+1}^{(j)} = R_{0,+}\big(a_2^{(j)}\big)$ and
$R_{-1}^{(j)} = R_{0,-}\big(a_2^{(j)}\big)$ of $V_2^{(j)}$
as in Lemma~\ref{lemma_dissection_proof_dihedral}.
By Lemmas~\ref{lemma_Cartesian_product_with_A_1}
and~\ref{lemma_dissection_proof_dihedral}(ii),
we have that
\begin{align*}
P(\Hf,L)
& =
\sum_{(\epsilon_1,\ldots,\epsilon_s)\in\{\pm 1\}^s}
\epsilon_1\cdots\epsilon_s
\left[L \cap \left(
S^{(1)} \times\cdots\times S^{(r)} \times
R_{\epsilon_1}^{(1)}\times\cdots\times R_{\epsilon_s}^{(s)}\right)\right] .
\end{align*}
\vanish{
and, for $s=0$, that
\begin{align*}
P(\Hf,L)
& =
\left[L \cap \left([0,2(a,e^{(1)})]\times\cdots\times[0,2(a,e^{(r)})]\right)\right],
\end{align*}
noting that $(a,e^{(j)})=(a_1^{(j)},e^{(j)})$ for every
$1 \leq j \leq r$.
}
Consider the polyhedron
\[
P =
[0,2 (a,e^{(1)}) e^{(1)}]
\times\cdots\times
[0,2 (a,e^{(r)}) e^{(r)}]
\times
R_0\big(a_2^{(1)}\big)
\times\cdots\times
R_0\big(a_2^{(s)}\big).\]
Then $P$ is the convex hull of the set $\{u(0) : u\in W_a\}$ 
by
definition of the polygons $R_0\big(a_2^{(j)}\big)$,
hence it is contained in $L$ and so is its subset
$S^{(1)} \times\cdots\times S^{(r)} \times
R_{\epsilon_1}^{(1)}\times\cdots\times R_{\epsilon_s}^{(s)}$ for every
$(\epsilon_1,\ldots,\epsilon_s)\in\{\pm 1\}^s$.
So we obtain
\begin{align*}
P(\Hf,L)
& =
\sum_{(\epsilon_1,\ldots,\epsilon_s)\in\{\pm 1\}^s}
\epsilon_1\cdots\epsilon_s
\left[S^{(1)} \times\cdots\times S^{(r)} \times
R_{\epsilon_1}^{(1)}\times\cdots\times R_{\epsilon_s}^{(s)}\right] .
\end{align*}
\vanish{
and that, for $s=0$,
\begin{align*}
P(\Hf,L)
& =
\left[[0,2(a,e^{(1)})]\times\cdots\times[0,2(a,e^{(r)})]\right].
\end{align*}
}
If $s=0$ this implies statement~(ii).
Suppose that $s\geq 1$.
By point (iii) of Lemma~\ref{lemma_dissection_proof_dihedral},
we know that
\begin{align*}
\left[S^{(1)} \times\cdots\times S^{(r)} \times
R_{\epsilon_1}^{(1)}\times\cdots\times R_{\epsilon_s}^{(s)}\right] 
= \: 
\left[S^{(1)} \times\cdots\times S^{(r)} \times
R_{+1}^{(1)}\times\cdots\times R_{+1}^{(s)}\right]
\end{align*}
\vanish{
\begin{align*}
&
\left[[0,2 (a,e^{(1)}) e^{(1)}] \times\cdots\times [0,2 (a,e^{(r)}) e^{(r)}] \times
R_{\epsilon_1}^{(1)}\times\cdots\times R_{\epsilon_s}^{(s)}\right] \\
= \: & 
\epsilon_1\cdots\epsilon_s
\left[[0,2 (a,e^{(1)}) e^{(1)}]\times\cdots\times[0,2 (a,e^{(r)}) e^{(r)}]\times
R_{+1}^{(1)}\times\cdots\times R_{+1}^{(s)}\right]
\end{align*}
}
for every $(\epsilon_1,\ldots,\epsilon_s)\in\{\pm 1\}^s$. 
As $\sum_{(\epsilon_1,\ldots,\epsilon_s)\in\{\pm 1\}^s}
\epsilon_1 \cdots \epsilon_s = 0$, this
finishes the proof of~(i).
\end{proof}

\begin{proof}[Proof of Theorem~\ref{theorem_pizza_dissection}.]
Statement~(ii) is exactly Lemma~\ref{lemma_pizza_for_A_1^rxI_2^s}(ii).
We now prove statement~(i), so we assume that $\Hf$ is not of type
$A_1^{n}$.
By Theorem~\ref{thm_pizza_abstract}, we have
\[P(\Hf,K+a)=\sum_{\varphi\in\Tt(\Phi)}\epsilon(\varphi)
P(\Hf_\varphi,K+a) . \]
By definition, any $2$-structure for $\Phi$ is of
type $A_1^r\times\prod_{k\geq 2}I_2(2^k)^{s_k}$
with
$\sum_{k\geq 2}s_k$ finite
and, as $W$ acts transitively on the set
of $2$-structures 
(Proposition~\ref{proposition_2_structures}(i)),
the integers $r$ and $s_k$, for $k\geq 2$, do not
depend on the $2$-structure but only on $\Phi$.
Also, by Proposition~\ref{proposition_2_structures}(ii), 
we have $\dim V=r+\sum_{k\geq 2}2s_k$, so we are in the situation
of Lemma~\ref{lemma_pizza_for_A_1^rxI_2^s}.
Suppose that $\sum_{k\geq 2}s_k\geq 1$. 
Then by Lemma~\ref{lemma_pizza_for_A_1^rxI_2^s}(i)
we have $P(\Hf_\varphi,K+a)=0$ for every $\varphi\in\Tt(\Phi)$
and hence $P(\Hf,K+a)=0$.
Assume now that $\sum_{k\geq 2}s_k=0$, that is, $s_{k} = 0$ for every $k$.
Statement~(ii) of the same lemma implies that
\begin{align}
P(\Hf,K+a)=\sum_{\varphi\in\Tt(\Phi)}\epsilon(\varphi)
\Bigg[\prod_{e\in\varphi\cap\Phi^+}(0,2(a,e)e]\Bigg].
\label{equation_sunshine}
\end{align}
This is an alternating sum of classes of half-open
rectangular parallelotopes in $V$.
So we can apply Theorem~\ref{thm_Bolyai-Gerwien_plus_epsilon} to
prove that $P(\Hf,K+a)=0$ in $K(V)$. We know that
$V_i(P(\Hf,K+a))=0$ if $0\leq i\leq n-1$ by
Lemma~\ref{lemma_intrinsic_volume_half_open_parallelotope},
so it remains to prove that $V_n(P(\Hf,K+a))=0$, that is,
that the alternating sum of the volumes of the parallelotopes
$\prod_{e\in\varphi\cap\Phi^+}(0,2(a,e)e]$ is equal to zero.
This follows from
Theorem~1.2 of the paper~\cite{EMR_pizza}.
However, we now give a direct
proof (that does not use analysis)
using the method of that corollary. 
Let $f:V\fl\R$ be the function defined by
\begin{align*}
f(a)
& =
\sum_{\varphi\in\Tt(\Phi)}\epsilon(\varphi) \prod_{e\in\varphi\cap\Phi^+} 2(a,e) .
\label{equation_sunshine}
\end{align*}
Note that $f$ is a polynomial homogeneous of degree $n$ on~$V$.
Furthermore equation~\eqref{equation_sunshine}
implies that
\[\Vol(P(\Hf,K+a))  
=f(a)\]
for every convex subset $K$ of $V$ of finite volume
that is stable by $W$ and every $a\in V$ such that $0\in K+a$.
The polynomial~$f$ satisfies $f(w(a))=\det(w)f(a)$ for every
$w\in W$ and every $a\in V$ (this is easy to see;
see, for example, Corollary~2.3 of~\cite{EMR_pizza}),
so it vanishes on every hyperplane of $\Hf$. But if $f\not=0$, then
the vanishing set of $f$ must be of degree at most $n$, 
which contradicts
the fact that, as $\Hf$ is not of type~$A_1^{n}$, we have $|\Hf|>n$.
Hence we must have $f=0$, and this gives the desired result.
\end{proof}

\begin{remark}
In the paper~\cite{HH}, 
J., M.\ D., J.\ K., A.\ D.\ and P.\ M.\ Hirschhorn proved that if
a circular pizza is cut into $4m$ slices by $2m$ cuts at equal angles
to each other and if $m$ people share the pizza by each taking every
$m$th slice then they receive equal shares. 
If $m=4$, Frederickson gives a dissection-based proof of this
fact on page~32 of~\cite{Fred}, and Proposition~9.1
of~\cite{EMR_pizza} generalizes the result to pizzas of more general
shapes. We cannot lift this result to the group~$K(V)$, because it
does not hold in that group. For example, if we consider the pizza
of Figure~\ref{figure_dissection_dihedral_2}, then it is not true
in general that the sums of the perimeters of the pizza pieces in
all the shares will be equal.

However, we can lift the generalization of the Hirschhorns's result to
the quotient $K_0(V)$ of the group $K(V)$ by the subgroup generated
by all the elements $[C]$ with $C\in\Cf(V)$ contained in a line of $V$.
More precisely, consider a Coxeter arrangement of type $I_2(2m)$ in
$\R^2$ with $m$ even and let~$W$ be the 
Coxeter group of this arrangement.
Let $T_0,\ldots,T_{4m-1}$ be the chambers of $\Hf$, indexed so that
that $T_i$ and $T_{i+1}$ share a wall.
Let $K\in \Cf(V)$ and $a\in V$. Suppose
that $K$ is stable by~$W$ and contains the convex hull of the
set $\{w(a) : w\in W\}$. Then the quantity
\[\sum_{i=0}^3[T_{r+mi}\cap (K+a)] \in K_0(V)\]
is independent of $0 \leq r \leq m-1$. 
(This implies the Hirschhorns's result even in the case where $k$
is odd: Just apply the previous statement with $m=2k$, and share the pizza
between $k$ people by giving the $p$th person the eight slices contained
in the chambers $T_{2p+mi}$ and $T_{2p+1+mi}$, for $0 \leq i \leq 3$.)

Let us prove this result.
Let $0 \leq r \leq m-2$.
We want to show that
\[\sum_{i=0}^3[T_{r+im}\cap(K+a)]=
\sum_{i=0}^3[T_{r+1+im}\cap(K+a)]\]
in $K_0(V)$.
Suppose that we know that
\begin{equation}\sum_{i=0}^3[T_{r+im}\cap R_+(a)\cap(K+a)]=
\sum_{i=0}^3[T_{r+1+im}\cap R_-(a)\cap(K+a)]
\label{equation_half_result}
\end{equation}
in $K_0(V)$, where $R_{+}(a)$ and $R_{-}(a)$ are given by
\[R_\pm(a)=(V-R_0(a))\cap\bigcup_{\substack{T\in\Tc(\Hf) \\ (-1)^T=\pm 1}} T.\]
Then it remains to see that
\[\sum_{i=0}^3[T_{r+im}\cap R_0(a)\cap(K+a)]=
\sum_{i=0}^3[T_{r+1+im}\cap R_0(a)\cap(K+a)]\]
in that same quotient.
But now all the regions appearing in the sums are polygons, so
the equality of the sums of their classes in $K_0(V)$ is equivalent
to the equality of the sums of their areas, by the
Bolyai-Gerwien Theorem; see~\cite[Section~5]{Bolt}. This last
equality follows either from Proposition~9.1 of~\cite{EMR_pizza},
or from the Hirschhorns's result. (The Hirschhorns only consider the
case of a circular pizza, but, by equation~\eqref{equation_half_result}, their
result for a circular pizza implies the result for the polygonal pizza
$R_0(a)$.)

We first suppose that $r$ is even.
To prove equation~\eqref{equation_half_result},
we suppose that $T_0,\ldots,T_{2m-1}$ denote the same
chambers as in the proof of
Lemma~\ref{lemma_dissection_proof_dihedral},
and we use the notation of that proof. 
In particular, for
$0\leq i\leq 2m-1$, the chamber~$T_{2m+i}$ is equal to the
region $R_{i,\epsilon}$, where $\epsilon$ is the sign $(-1)^{i}$. 
We have
$T_r\cap R_+(a)=S_{r/2,+}\cup
R_{r+1,+}$, $T_{r+m}\cap R_+(a)=S_{(r+m)/2,+}\cup R_{r+m+1,+}$,
$T_{r+2m}\cap R_+(a)=T_{r+2m}=R_{r,+}$ and $T_{r+3m}\cap R_+(a)=
T_{r+3m}=R_{r+m,+}$. On the other hand, we have
$T_{r+1}\cap R_-(a)=S_{r/2,-}\cup
R_{r,-}$, $T_{r+1+m}\cap R_-(a)=S_{(r+m)/2,-}\cup R_{r+m,-}$,
$T_{r+1+2m}\cap R_-(a)=T_{r+1+2m}=R_{r+1,-}$ and $T_{r+1+3m}\cap R_-(a)=
T_{r+1+3m}=R_{r+m+1,-}$. This implies equation~\eqref{equation_half_result}.

We now consider the case where $r$ is odd.
We again suppose that $T_0,\ldots,T_{2m-1}$ denote the same
chambers as in the proof of
Lemma~\ref{lemma_dissection_proof_dihedral} and the notation of that lemma, 
but we use a different
dissection, that is illustrated in Figure~\ref{figure_dissection_dihedral_3}
in the case $m=4$.
\begin{figure}
\centering
\includegraphics[width=\textwidth]{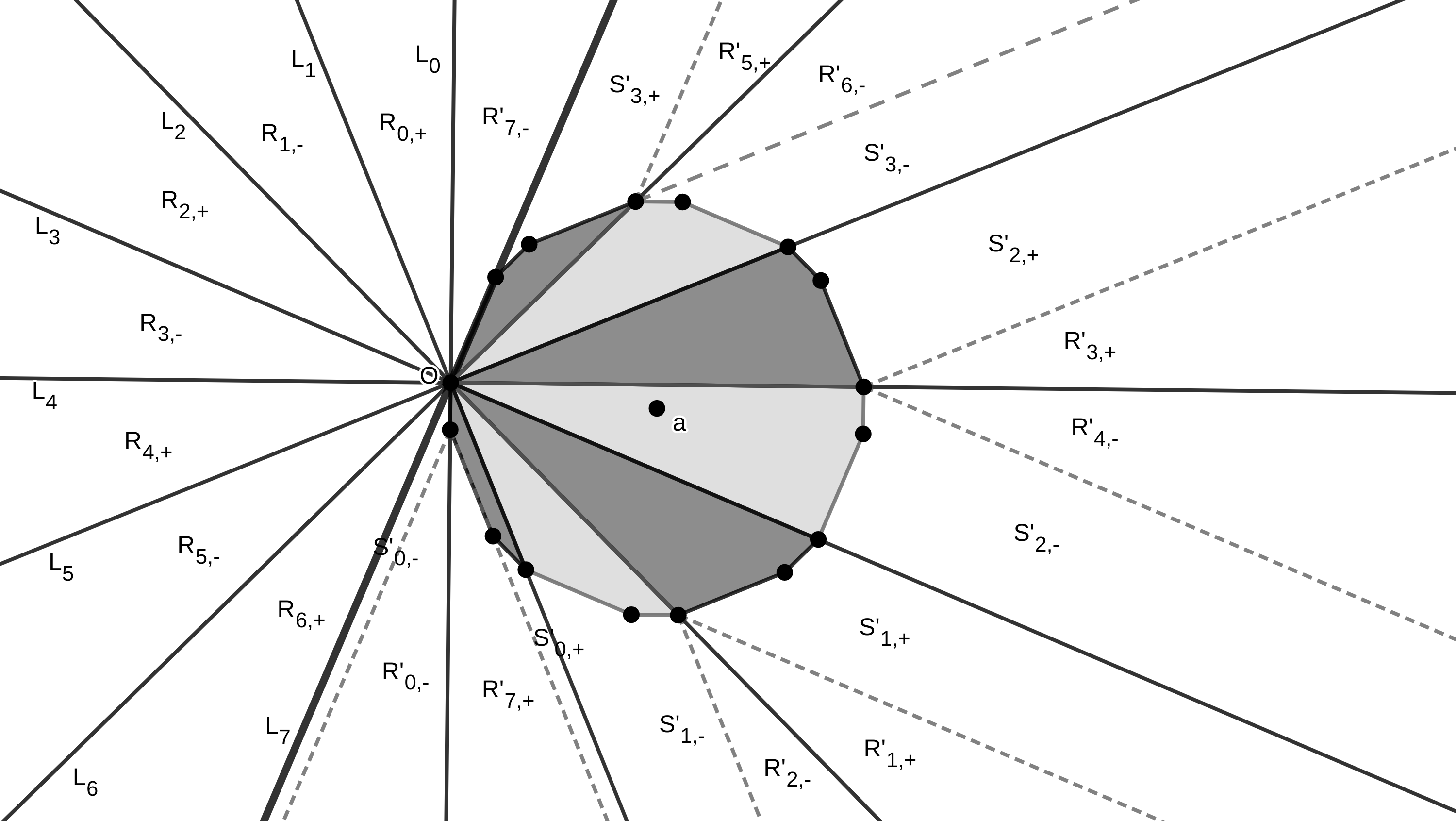}
\caption{
The dissection that we use for odd $r$ in 
the case of $I_2(8)$.}
\label{figure_dissection_dihedral_3}
\end{figure}
For $0\leq i\leq 2m-2$, we consider the same region
$R_{i,\pm}$ as
in the proof of Lemma~\ref{lemma_dissection_proof_dihedral}, but
we denote by $R'_{2m-1,-}$ the chamber $T_{2m-1}$.
For every $1 \leq j \leq m-1$, we denote
by $R'_{2j-1,+}$ (respectively, $R'_{2j,-}$) the orthogonal reflection of 
$R_{2j,+}$ (respectively, $R_{2j+1,-}$) in the line $L_{2j}^\perp+a$; 
note that $R'_{2j-1,+}\subset T_{2j}$ and $R'_{2j,-}\subset T_{2j-1}$.
We also denote
by $R'_{0,-}$ (respectively, $R'_{2m-1,+}$) the orthogonal reflection of 
$R_{0,+}$ (respectively, $R'_{2m-1,-}$) in the line $L_{m}^\perp+a$.
For $1 \leq j \leq m-1$ again, we denote by
$S'_{j,-}$ the interior of $T_{2j-1} -(R_0(a)\cup R'_{2j,-})$
and by $S'_{j,+}$ the interior of
$T_{2j} - (R_0(a)\cup R'_{2j-1,+})$.
Finally, we denote by $S'_{0,-}$ the interior of
$T_{4m-1}-(R_0(a)\cup R'_{0,-})$ and $S'_{0,+}$ the interior of
$T_0-(R_0(a)\cup R'_{2m-1,+})$.

For
$0\leq i\leq 2m-2$, the chamber~$T_{2m+i}$ is then equal to the
region $R_{i,\epsilon}$, where $\epsilon$ is the sign $(-1)^{i}$;
also, the chamber $T_{2m-1}$ is equal to $R'_{2m-1,-}$. 
We have
$T_r\cap R_-(a)=S_{(r+1)/2,-}\cup
R_{r+1,-}$, $T_{r+m}\cap R_-(a)=S_{(r+m+1)/2,-}\cup R_{r+m+1,-}$,
$T_{r+2m}\cap R_-(a)=T_{r+2m}=R_{r,-}$ and $T_{r+3m}\cap R_-(a)=
T_{r+3m}=R_{r+m,-}$. On the other hand, we have
$T_{r+1}\cap R_+(a)=S_{(r+1)/2,+}\cup
R_{r,+}$, $T_{r+1+m}\cap R_+(a)=S_{(r+1+m)/2,+}\cup R_{r+m,+}$,
$T_{r+1+2m}\cap R_+(a)=T_{r+1+2m}=R_{r+1,+}$ and $T_{r+1+3m}\cap R_+(a)=
T_{r+1+3m}=R_{r+m+1,+}$. This implies equation~\eqref{equation_half_result}.

As in Frederickson's article, there should also
be a dissection-based proof of the equality of areas
that we use to finish the proof, but we
were not courageous enough to look for it.
\end{remark}

\section{The Bolyai-Gerwien Theorem}
\label{section_Bolyai-Gerwien_plus_epsilon}

The classical Bolyai-Gerwien Theorem states that two polygons are
scissors congruent if and only if they have the same area. There is also
a well-known generalization in higher dimensions that applies
to parallelotopes; it follows from the characterization of translational
scissors congruences in arbitrary dimensions, and was proved independently
by Jessen-Thorup and Sah; see 
the beginning of Section~7 of~\cite{Jessen_Thorup} or
Theorem~1.1 in Chapter~4 of~\cite{Sah}. 
In this section, we state a slight refinement of this
generalization, Theorem~\ref{thm_Bolyai-Gerwien_plus_epsilon},
that keeps track of lower-dimensional
faces; in other words, we do not want to ignore the boundaries.

As in the previous sections, let $V$ be an $n$-dimensional real
vector space with an inner product $(\cdot,\cdot)$.
If $(v_1,\ldots,v_r)$ is a linearly independent list of elements of $V$,
we define the parallelotope
\[P(v_1,\ldots,v_r)=\left\{\sum_{i=1}^r a_i v_i : 0\leq a_i\leq 1\right\}.\]
We denote by
$\Pf(V)$ the set of all convex polytopes in $V$ (including lower-dimensional
ones) and by
$\Zf(V)$ the subfamily of polytopes that are translates of parallelotopes
of the form $P(v_1,\ldots,v_r)$.
The set $\Pf(V)$ satisfies the
conditions of Definition~\ref{def_K(V)}, 
so we can define an abelian group
$K_\Pf(V)$ as in that definition. We denote by
$K_\Zf(V)$ the subgroup of $K_\Pf(V)$ generated by the classes~$[P]$
for $P\in\Zf(V)$.
Remark~\ref{remark_[C]_for_more_general_C} implies that, if
$\Pfe(V)$ is 
the relative Boolean algebra generated
by $\Pf(V)$, 
then we can define the class
$[P]$ in $K_\Pf(V)$
of any element $P$ in
$\Pfe(V)$.
We denote by $\Zfe(V)$ the set of elements $P$ of $\Pfe(V)$ such that
$[P]\in K_\Pf(V)$ is in the subgroup $K_\Zf(V)$. For example,
the set $\Zfe(V)$ contains $\Zf(V)$, and it also contains all half-open
parallelotopes.

Recall $V_0,\ldots,V_n$ denotes the intrinsic volumes on $V$;
see~\cite[Section~4.2]{Schneider}. These are valuations on the set of
all compact convex subsets of $V$, and in particular on
$\Pf(V)$, so they induce morphisms of
groups from $K_\Pf(V)$ to $\R$, which we still denote
by $V_0,\ldots,V_n$. Note that $V_0$ is the Euler-Poincar\'e characteristic
with compact support,
so the image of $K_\Pf(V)$ is~$\Z$.

The main result of this section is the following isomorphism.

\begin{theorem}
The morphism $(V_0,V_1,\ldots,V_n):K_\Zf(V)\fl\Z\times\R^n$ is an
isomorphism.
In particular, if $P,P'\in
\Zfe(V)$ are such that $V_i(P)=V_i(P')$ for every $0 \leq i \leq n$,
then $[P]=[P']$ in $K_\Zf(V)$.
\label{thm_Bolyai-Gerwien_plus_epsilon}
\end{theorem}

We will give the proof of the theorem at the end of this section.

\begin{corollary}
Suppose that $\dim(V)=2$. Then the triple $(V_0,V_1,V_2)$ induces an isomorphism
from $K_\Pf(V)$ to $\Z\times\R^2$. In particular, if $P,P'$ are two elements
of $\Pfe(V)$, then their classes in $K_\Pf(V)$ are equal if and only
if $V_i(P)=V_i(P')$ for all $0 \leq i \leq 2$.
\label{cor_Bolyai-Gerwien_plus_epsilon}
\end{corollary}

\begin{proof}
By Theorem~\ref{thm_Bolyai-Gerwien_plus_epsilon} it suffices to prove that $K_\Pf(V)=K_\Zf(V)$.
As points and segments are parallelotopes, 
it suffices to prove that every polygon $P$ is
scissors congruent to a parallelogram, which follows from
the Bolyai-Gerwien Theorem (see~\cite[Section~5]{Bolt}).
\end{proof}

\begin{remark}
\begin{itemize}
\item[(1)] If we take the quotient $K_{\Pf,0}(V)$
of $K_\Pf(V)$ by the subgroup generated
by the classes of lower-dimensional polytopes, then two polytopes
have the same class in $K_{\Pf,0}(V)$ if and only if they are
scissors congruent, and the Bolyai-Gerwien Theorem
(see~\cite[Section~5]{Bolt}), 
says that, if $\dim(V)=2$, the area $V_2$ induces an isomorphism 
from $K_{\Pf,0}(V)$ to $\R$. 

\item[(2)] If $\dim(V)\geq 3$, then $K_\Zf(V)\not= K_\Pf(V)$.
Otherwise, every element of $\Pf(V)$ 
of positive volume would be scissors congruent to an element of
$\Zf(V)$, hence to a cube, and this is not true by the negative solution
to Hilbert's third problem (see for example~\cite{Bolt}).

\end{itemize}
\end{remark}

Let $\Zf'(V)$ be the set of translates of parallelotopes of the form
$P(v_1,\ldots,v_r)$, for $(v_1,\ldots,v_r)$ a linearly independent
family of elements of $V$ such that $r\leq n-1$, and let
$K'_\Zf(V)$ be the subgroup of $K_\Zf(V)$ generated by the classes
of elements of $\Zf'(V)$. We also write $\Zfep(V)$ for the
set of $P\in\Zfe(V)$ such that $[P]\in K'_\Zf(V)$.

If $P,P'\in\Zf(V)$, we write $P\sim P'$ if
there exist 
$P_1,\ldots,P_r,Q_1,\ldots,Q_s\in\Pfe(V)$,
$a_1,\ldots,a_r$, $b_1,\ldots,b_s\in V$ and
$R,R'\in\Zfep(V)$ such that
\[P\sqcup Q_1\sqcup\cdots\sqcup Q_s=R\sqcup P_1\sqcup\cdots\sqcup P_r
\text{ and }
P'\sqcup(Q_1+b_1)\sqcup\cdots\sqcup(Q_s+b_s)=R'\sqcup(P_1+a_1)\sqcup\cdots
\sqcup(P_r+a_r).\]
It is not hard to see that this is an equivalence relation, and that
equivalent parallelotopes have the same volume.

\begin{lemma}
Let $W,W'$ be subspaces of $V$ such that $V=W\times W'$.
We do not assume that $W$ and~$W'$ are orthogonal.
Let $P,P'\in\Zf(W)$ and $S,S'\in\Zf(W')$ such that
$P\sim P'$ and $S\sim S'$.
Then the relation $P\times S\sim P'\times S'$ holds.
\label{lemma_product_decomposition}
\end{lemma}
\begin{proof}
As $\sim$ is transitive and as $W$ and $W'$ play symmetric roles, it
suffices to treat the case where $S=S'$. We choose
$P_1,\ldots,P_r,Q_1,\ldots,Q_s\in\Pfe(W)$,
$a_1,\ldots,a_r,b_1,\ldots,b_s\in W$ and
$R,R'\in\Zfep(W)$ such that
\[P\sqcup Q_1\sqcup\cdots\sqcup Q_s=R\sqcup P_1\sqcup\cdots\sqcup P_r
\text{ and }
P'\sqcup(Q_1+b_1)\sqcup\cdots\sqcup(Q_s+b_s)
=
R'\sqcup(P_1+a_1)\sqcup\cdots\sqcup(P_r+a_r).\]
Then
\[(P\times S)\sqcup(Q_1\times S)\sqcup\cdots\sqcup(Q_s\times S)
=(R\times S)\sqcup (P_1\times S)\sqcup\cdots\sqcup 
(P_r\times S)\]
and
\[(P'\times S)\sqcup((Q_1\times S)+b_1)\sqcup\cdots\sqcup((Q_s\times S)+b_s)
=
(R'\times S)\sqcup ((P_1\times S)+a_1)\sqcup\cdots\sqcup ((P_r\times S)+a_r).\]
As $R\times S$ and $R'\times S$ are in $\Zfep(V)$, this implies that
$P\times S\sim P'\times S$.
\end{proof}

For two real numbers $a, b \in \R$
recall that the half open interval
is given by $(a,b] = \{t\in\R:a<t\leq b\}$ and
the closed interval by
$[a,b] = \{t\in\R:a\leq t\leq b\}$.

\begin{lemma}
Let $(v_1,\ldots,v_n)$ and $(w_1,\ldots,w_n)$ be bases of
$V$, and let $P=P(v_1,\ldots,v_n)$ and $P'=P(w_1,\ldots,w_n)$.
Then $V_n(P)=V_n(P')$ if and only if there exists an isometry $g$ of
$V$ such that $P\sim g(P')$.

In particular, if $V=\R^n$, then, for every $P\in\Zf(V)$, 
the classes of $P$ and of $(0,1]^{n-1}\times (0,V_n(P)]$ in $K_\Zf(V)$ are
equal modulo $K'_\Zf(V)$.

\label{lemma_rectangular_parallelotopes}
\end{lemma}

\begin{proof}
We already know that $V_n(P)=V_n(P')$ if $P\sim g(P')$ with $g$ an
isometry of $V$.
We prove the converse by induction on~$n$. 
It suffices to show that, for every basis $(v_1,\ldots,v_n)$ of
$V$, there exists an orthonormal basis $(e_1,\ldots,e_n)$ of $V$ and
$a\in\R_{\geq 0}$ such that $P(v_1,\ldots,v_n)\sim 
P(e_1,\ldots,e_{n-1},a\cdot e_n)$;
we then must have 
$a=V_n(P(v_1,\ldots,v_n))$.
There is nothing to prove if $n=0$, and the claim is clear if
$n=1$.
Suppose that $n=2$. 
The classical proofs that two parallelograms that
have the same basis and the same height are scissors congruent
and that rectangles that have parallel sides and the same area
are scissors congruent
use only translations to move the pieces of the decompositions
(see for example~\cite[Proposition~35]{Euclid} and
Figure~30 on page~52 of~\cite{Bolt}; we reproduce the
relevant decompositions in Figures~\ref{figure_dissection_Bolyai-Gerwien1}
and~\ref{figure_dissection_Bolyai-Gerwien2}).
\begin{figure}
\begin{center}
\begin{tikzpicture}[scale = 1]
\filldraw[gray!10] (0,0) -- (0.7,3) -- (2.3,3) -- cycle;
\filldraw[gray!10] (6,0) -- (6.7,3) -- (8.3,3) -- cycle;
\filldraw[gray!35] (0,0) -- (2.3,3) -- (6.7,3) -- (6,0) -- cycle;
\draw[thick] (0,0) -- (6,0) -- (6.7,3) -- (0.7,3) -- (0,0) -- (2.3,3) -- (8.3,3) -- (6,0);
\filldraw (0,0) circle (0.5mm);
\filldraw (6,0) circle (0.5mm);
\filldraw (2.3,3) circle (0.5mm);
\filldraw (8.3,3) circle (0.5mm);
\filldraw (0.7,3) circle (0.5mm);
\filldraw (6.7,3) circle (0.5mm);
\end{tikzpicture}
\end{center}
\caption{Proof of Lemma~\ref{lemma_rectangular_parallelotopes}.}
\label{figure_dissection_Bolyai-Gerwien1}
\end{figure}
\begin{figure}
\begin{center}
\begin{tikzpicture}[scale = 1]
\filldraw[gray!10] (0,0) rectangle (8,3);
\filldraw[gray!10] (0,0) rectangle (6,4);
\filldraw[gray!35] (6,1) -- (6,3) -- (2,3) -- cycle;
\draw[thick] (0,0) rectangle (8,3);
\draw[thick] (0,0) rectangle (6,4);
\draw[thick,-] (8,0) -- (0,4);
\filldraw (0,0) circle (0.5mm);
\filldraw (6,0) circle (0.5mm);
\filldraw (8,0) circle (0.5mm);
\filldraw (0,3) circle (0.5mm);
\filldraw (0,4) circle (0.5mm);
\filldraw (6,4) circle (0.5mm);
\filldraw (8,3) circle (0.5mm);
\filldraw (6,1) circle (0.5mm);
\filldraw (2,3) circle (0.5mm);
\end{tikzpicture}
\end{center}
\caption{Proof of Lemma~\ref{lemma_rectangular_parallelotopes}.}
\label{figure_dissection_Bolyai-Gerwien2}
\end{figure}
As the boundaries of the polygons that we ignore when we are
talking about scissors congruence are in $\Zfep(V)$ when
$\dim(V)=2$, this gives the claim.

Suppose that $n\geq 3$. 
Let $(v_1,\ldots,v_n)$ be a basis of $V$, and let
$P=P(v_1,\ldots,v_n)$. By the claim for $n=2$ and
Lemma~\ref{lemma_product_decomposition}, there exists an
orthonormal basis $(e_1,e_2)$ of $\Span(v_1,v_2)$ and
$a\in\R_{\geq 0}$ such that $P\sim P(e_1,a\cdot e_2,v_3,\ldots,v_n)$.
Applying the $n=2$ case in $\Span(e_1,v_3)$ and
Lemma~\ref{lemma_product_decomposition}, we can find 
$v'_3\in\Span(e_1,v_3)$ orthogonal to $e_1$ such that
$P\sim P(e_1,a\cdot e_2,v'_3,v_4,\ldots,v_n)$. Now applying
the $n=2$ case in $\Span(e_2,v'_3)$, noting that this space
is orthogonal to $e_1$, and using Lemma~\ref{lemma_product_decomposition},
we can find a unit vector $e_3\in\Span(e_2,v'_3)$ that is orthogonal
to $e_1$ and $e_2$ and $b\in\R_{\geq 0}$ such that
$P\sim P(e_1,e_2,b\cdot e_3,v_4,\ldots,v_n)$. 
Continuing in this way, we finally obtain the claim.

We prove the last sentence of the lemma. If $P$ is a translate
of $P(v_1,\ldots,v_k)$ with $k\leq n-1$, then the classes of
$P$ and of $(0,1]^{n-1}\times(0,V_n(P)]$ are both in $K'_\Zf(V)$.
Suppose that $P$ is a translate of $P(v_1,\ldots,v_n)$, with
$(v_1,\ldots,v_n)$ a basis of $V$.
By the first assertion,
there exists an isometry $g$ of $V$ such that $g\cdot P\sim
[0,1]^{n-1}\times[0,V_n(P)]$, so the classes of $P$ and of $[0,1]^{n-1}\times
[0,V_n(P)]$ in $K_\Zf(V)$ are equal modulo $K'_\Zf(V)$. As the difference
between the classes of $[0,1]^{n-1}\times[0,V_n(P)]$ and
of $(0,1]^{n-1}\times(0,V_n(P)]$ is in $K'_\Zf(V)$, the result follows.
\end{proof}

\begin{lemma}
\begin{enumerate}
\item Let $H$ be a hyperplane of $V$. The inclusion $\Zf(H)\subset
\Zf(V)$ induces a morphism $K_\Zf(H)\fl K_\Zf(V)$ whose image
is $K'_\Zf(V)$.
\footnote{In fact, it follows from 
Theorem~\ref{thm_Bolyai-Gerwien_plus_epsilon} that $K_\Zf(H)\fl
K_\Zf(V)$ is injective, so we get an isomorphism from $K_\Zf(H)$ to $K'_\Zf(V)$.}

\item The subgroup $K'_\Zf(V)$ is the kernel of the morphism
$V_n:K_\Zf(V)\fl\R$.

\end{enumerate}
\label{lemma_flat_parallelotopes}
\end{lemma}

\begin{proof}
\begin{enumerate}
\item The existence of the morphism $K_\Zf(H)\fl K_\Zf(V)$ is
clear, as well as the fact that its image is contained in
$K'_\Zf(V)$. Conversely, any translate of a $P(v_1,\ldots,v_k)$
with $k\leq n-1$ can be moved by an affine isometry to lie in
$H$, so its class is in the image of $K_\Zf(H)$.

\item We may assume that $V=\R^n$.
Any polytope in $\Zf'(V)$ has volume zero, so
$K'_\Zf(V)$ is included in the kernel of $V_n$. We prove the reverse
inclusion. Let $x$ be an element of $\Ker V_n$, and write
$x=\sum_{i=1}^r\alpha_i[P_i]$, with $\alpha_i\in\{\pm 1\}$ and
$P_i\in\Zf(V)$. 
We want to show that $x\in K'_\Zf(V)$. 
By Lemma~\ref{lemma_rectangular_parallelotopes}, for every $1 \leq i \leq r$,
the class of $P_i$ is equal to the class
of $(0,1]^{n-1}\times (0,V_n(P_i)]$ modulo $K'_\Zf(V)$.
So $x$ is equal modulo $K'_\Zf(V)$ to the sum
\[\sum_{i=1}^{r} \alpha_i[(0,1]^{n-1}\times (0,V_n(P_i)]]=
[(0,1)^{n-1}\times(0,V_+]]-[(0,1]^{n-1}\times(0,V_-]],\]
where $V_{\pm}=\sum_{1\leq i\leq r,\ \alpha_i=\pm 1}\Vol(P_i)$.
As $V_+-V_-=V_n(x)=0$ by assumption, we conclude that
$x\in K'_\Zf(V)$.
\qedhere
\end{enumerate}
\end{proof}

\begin{proof}[Proof of Theorem~\ref{thm_Bolyai-Gerwien_plus_epsilon}]
Let $V_*=(V_0,V_1,\ldots,V_n):K_\Zf(V)\fl\Z\times\R^n$.
Then the morphism $V_*$ sends the class of a point
to $(1,0,\ldots,0)$, so its image contains the factor~$\Z$.
Denote by $(e_1,\ldots,e_n)$ the canonical basis of $\R^n$.
If $i\in\{1,\ldots,n\}$ and
$a\in\R_{\geq 0}$, then 
by Lemma~\ref{lemma_intrinsic_volume_half_open_parallelotope}
$V_*$ sends the class of the half-open rectangular
parallelotope $(0,1]^{i-1}\times(0,a]\times\{0\}^{n-i}$ to
$(0,a\cdot e_i)\in\Z\times\R^n$,
so the image of~$V_*$ contains $\R\cdot e_i$.
This shows that $V_*$ is surjective.

We now prove the injectivity of $V_*$
by induction on $\dim(V)$. If $\dim(V)=0$, the
result is clear. Suppose that $\dim(V)>0$ and that we know the
result for spaces of smaller dimension. Let $x\in K_\Zf(V)$ such
that $V_i(x)=0$ for $0\leq i\leq n$. In particular, we have
$V_n(x)=0$, so $x\in K'_\Zf(V)$ by Lemma~\ref{lemma_flat_parallelotopes}(ii).
Let $H$ be a hyperplane of $V$. Then $x$ is in the image of
the morphism $K_\Zf(H)\fl K_\Zf(V)$ 
by Lemma~\ref{lemma_flat_parallelotopes}(ii); choose a preimage
$y\in K_\Zf(H)$ of $x$. If $0\leq i\leq n-1$, then we have
$V_i(y)=V_i(x)$ because intrinsic volumes do not depend on
the dimension of the ambient space (see the top of page~214
of~\cite{Schneider}), so $V_i(y)=0$. It follows from the
induction hypothesis that $y=0$, and we conclude that $x=0$.
\end{proof}

\section*{Acknowledgements}

The authors thank Dominik Schmid
for introducing them to the Pizza Theorem, and Ramon van Handel
for dispelling some of their misconceptions and lending them
a copy of \cite{Schneider}.
They made extensive use of Geogebra to understand the $2$-dimensional
situation and to produce some of the figures.
This work was partially 
supported by the LABEX MILYON (ANR-10-LABX-0070) of Universit\'e 
de Lyon, within the program ``Investissements d'Avenir'' (ANR-11-IDEX-0007)
operated by the French National Research Agency (ANR),
and by Princeton University.
This work was also partially supported by grants from the
Simons Foundation
(\#429370 and \#854548 to Richard~Ehrenborg
and \#422467 to Margaret~Readdy). 
The third author was also supported by
NSF grant DMS-2247382.

\printbibliography

\end{document}